\documentclass[11pt,twoside]{article}
\usepackage{amssymb}
\usepackage{amsthm}
\usepackage{amsmath}

\usepackage{graphicx}
\usepackage[labelfont=bf]{caption}
\theoremstyle{plain}
\newtheorem{thm}{Theorem}[section] 
\newtheorem{lem}[thm]{Lemma}
\newtheorem{prop}[thm]{Proposition}
\newtheorem{cor}[thm]{Corollary} 

\theoremstyle{definition} 
\newtheorem{defn}[thm]{Definition}

\newtheorem{exmp}[thm]{Example}

\theoremstyle{remark} 
\newtheorem*{rem}{Remark}

\DeclareMathOperator*{\mycup}{\cup}

\begin{document}

\title{Splittings of Non-Finitely Generated Groups}
\date{last updated 9/12/2011}
\author{Robin M. Lassonde}

\maketitle


\begin{abstract}
In geometric group theory one uses group actions on spaces to gain information about groups.
One natural space to use is the Cayley graph of a group.
The Cayley graph arguments that one encounters
tend to require local finiteness,
and hence finite generation of the group.
In this paper, I take the theory of intersections of splittings of finitely generated groups
(as developed by Scott, Scott-Swarup, and Niblo-Sageev-Scott-Swarup),
and rework it to remove finite generation assumptions.
Whereas the aforementioned authors relied on the local finiteness of the Cayley graph, I capitalize on the Bass-Serre trees for the splittings.
\end{abstract}


\tableofcontents


\listoffigures


\section{History}

In this paper, we investigate crossing patterns of group splittings.
A splitting of a group $G$ is an algebraic generalization of
a codimension-$1$ submanifold of a manifold whose fundamental group is $G$.
In~\cite{Scott1998}, Scott gave a definition of the intersection number of two
almost invariant subsets of a finitely generated group,
and proved that the definition is symmetric.
Shortly after, Scott and Swarup showed that if two splittings of a finitely generated group
over finitely generated subgroups have intersection number zero,
then the two splittings are compatible (see~\cite{ScottSwarup2000}).
The same authors further developed these concepts in~\cite{ScottSwarup2003} and in~\cite{NibloSageevScottSwarup2005}
to construct an algebraic regular neighborhood for any finite collection of splittings of a finitely generated group.
The authors were concerned only with almost invariant subsets of finitely generated groups,
and often restricted to the case where the stabilizer of each almost invariant set is finitely generated.
Here, we use the properties of splittings to remove the finite generation assumptions
in the case where the almost invariant sets come from splittings,
and see how far we can push the theory before it falls apart.
The answer is: quite far.
Instead of using the Cayley graph (which is only useful when $G$ is finitely generated), we turn to the Bass-Serre trees for the splittings.

Here I review history leading up to the theory described above.

\subsection{Ends}

In 1931, Freudenthal~\cite{Freudenthal1931} defined the number of ends of a topological space.
Roughly speaking, the ends of a space are the space's ``connected components at infinity.''
To count the number of ends of a locally finite CW complex,
remove a finite number of open cells, and count the number of infinite components remaining.
The number of ends is the supremum over all such removals.
In 1944, Hopf realized that the number of ends of the Cayley graph of a finitely generated group
does not depend on the choice of generating set;
hence one can define the number of ends of a finitely generated group
to be the number of ends of its Cayley graph (see~\cite{Hopf1944}).
Slightly more is true: if a finitely generated group acts cocompactly on a locally finite space,
then the number of ends of the space is the same as the number of ends of a group.
Several years later, Specker gave a purely algebraic definition for the number of ends of any group
(see~\cite{Specker1950}).
Usually when we think of the number of ends of a group, we are thinking of the geometric interpretation (which only works for finitely generated groups).
However, it is good to know that the definition can be extended to non-finitely generated groups.

\subsection{Splittings}

Group splittings were defined around the same time as ends.
Schreier introduced amalgamated free products in 1927 (see~\cite{Schreier1927}),
and two decades later, Higman-Neumann-Neumann introduced HNN extensions (see~\cite{HNN1949}).
Both amalgamated free products and HNN extensions were initially described
in terms of normal forms for words.
In 1977, Serre discovered that splittings can be described as group actions on trees.
This topic is known as ``Bass-Serre theory'' (see~\cite{Serre1977} and~\cite{Serre1980}).
Immediately after, Scott and Wall noted that one can use the Seifert-van-Kampen Theorem
to realize any graph of groups as a graph of spaces (see~\cite{ScottWall1979}).
From Scott's point of view, there is no reason to distinguish between amalgamated free products and HNN extensions,
so he called both ``splittings.''

\subsection{Stallings' Theorem}

A decade before Serre's discovery that group actions on trees correspond to splittings,
Stallings made a connection between splittings and ends of groups.
Stallings' theorem states that a finitely generated group $G$ has at least two ends if, and only if, $G$ splits over a finite subgroup
(see~\cite{Stallings1971} and~\cite{Stallings1968}).
He also showed that groups of cohomological dimension one are free.
Swan extended Stallings' results to non-finitely generated groups (see~\cite{Swan1969}).

\subsection{Almost Invariant Sets}

Also in the 1970's, Cohen coined the term ``almost invariant set'' over the trivial subgroup, as a way of keeping track of the ends of a group
(see~\cite{Cohen1970}).
Twenty years earlier Specker used almost invariant sets in his paper~\cite{Specker1950};
however, this fact has been entirely overlooked in the history of almost invariant sets,
in part because Specker was interested in something more general.
Cohen also observed that a subset of a finitely generated group is almost invariant if, and only if,
the subset has finite coboundary in the Cayley graph of the group.
Houghton formally defined the number of ends of a pair $(G,H)$ of groups (see~\cite{Houghton1974}),
and one can make a similar observation that
a subset of a finitely generated group is $H$-almost invariant if, and only if,
the subset has finite coboundary in the Cayley graph quotiented out by $H$.

Dunwoody used Bass-Serre theory and Cohen's almost invariant sets to
produce a beautiful geometric proof of Stallings' theorem in the finitely generated case,
in which one takes an ``end'' of $G$ and uses it to directly construct a Bass-Serre tree (see Section~\ref{SUBdunwoody}, \cite{Dunwoody1979} and Section~$6$ of~\cite{ScottWall1979}).
In 1995, Sageev showed how to construct a $CAT(0)$ cubical complex from an almost invariant set (see~\cite{Sageev1995}).
In~\cite{Niblo2004}, Niblo used Sageev's construction to produce another geometric proof of Stallings' theorem.

\section{What's Proved in this Paper}
\label{Intro}

For an explanation of the concepts used in this section, see Section~\ref{Preliminaries}.

A splitting of a group $G$ is a one-edged graph of groups structure for $G$.
In this paper, it is convenient to talk about everything in terms of $G$-trees, to entirely avoid mentioning graphs of groups.
For example, looking at graphs of groups, it is not clear how to define an isomorphism of splittings.
On the other hand, using the $G$-tree definition of ``splitting,'' clearly we should define two splittings to be isomorphic precisely when their $G$-trees are isomorphic.

In~\cite{Scott1998}, Scott defined ``$X$ crosses $Y$,'' where $X$ is an $H$-almost invariant subset of $G$, and $Y$ is a $K$-almost invariant subset of $G$.
He proved that if $G$ is finitely generated,
then $X$ crosses $Y$ if, and only if, $Y$ crosses $X$.
By counting the number of group elements $g \in G$ such that $gX$ crosses $Y$,
Scott gave a well-defined, symmetric intersection number of $X$ and $Y$.
If $X$ and $Y$ come from splittings of $G$,
this gives a symmetric intersection number of the two splittings.
In this paper, I prove that if $Y$ arises from a splitting of $G$, and if $X$ crosses $Y$, then $Y$ crosses $X$
(see Proposition~\ref{SYMMETRIC}), without any assumption of finite generation.
In particular, for any group (not necessarily finitely generated),
the intersection number of two splittings is well-defined.
The key argument used to prove this is laid out in Lemma~\ref{HANDKFINITE}.

Also in~\cite{Scott1998}, Scott proved that if $X$ is an $H$-almost invariant subset of $G$, and
$Y$ is a $K$-almost invariant subset of $G$,
such that all of $G$, $H$ and $K$ are finitely generated, then the intersection number of $X$ and $Y$ is finite.
I give two examples showing that the assumption that $G$ be finitely generated is crucial:
it is possible for two standard $\{1\}$-almost invariant sets to have infinite intersection number (Example~\ref{INFINTERNUMBER}), and also possible for a $\{1\}$-almost invariant set to have infinite self-intersection number (Example~\ref{INFSELFINTER}).

In~\cite{ScottSwarup2000}, Scott and Swarup proved that
given a finite collection of splittings of a finitely generated group over finitely generated subgroups,
if the splittings have pairwise intersection number zero,
then they are compatible.
In this paper, I generalize Scott and Swarup's ``intersection number zero implies compatible'' theorem
to work without any finite generation assumptions.
One must instead make a sandwiching assumption.
Sandwiching is automatic if none of the splittings is a trivially ascending HNN extension (see Corollary~\ref{TRIVHNN}),
and sandwiching is necessary in order for the theorem to hold (see Section~\ref{NECESSARY}).
My proof mirrors the proof in~\cite{ScottSwarup2000},
replacing all coboundary arguments by new arguments using $G$-trees for splittings.
In particular:
\begin{itemize}
	\item
	Scott and Swarup used Cayley graph arguments to show that
	``almost inclusion'' defines a partial order on the set of all translates of all the almost invariant sets
	(arising from the splittings with intersection number zero) and their complements.
	I show that the fact that almost inclusion defines a partial order
	can be deduced directly once one has symmetry of crossing (see Corollary~\ref{PARTIALORDER}).
	\item
	Scott and Swarup used Cayley graph arguments to prove interval finiteness.
	Their arguments require both $G$ and the associated subgroups to be finitely generated.
	I show how to deduce interval finiteness directly from sandwiching (see Proposition~\ref{INTERVALS}).
\end{itemize}

For the rest of this section, let $\{X_j | j = 1 \ldots n\}$ be a finite collection of $H_j$-almost invariant subsets of a group $G$,
and let $\leq$ denote almost inclusion.

In~\cite{Sageev1995}, Sageev constructed a $CAT(0)$ cubical complex from the partially ordered set $(\Sigma, \subset)$.
In~\cite{NibloSageevScottSwarup2005}, Niblo, Sageev, Scott and Swarup generalized
Sageev's construction, using the partial order of almost inclusion on $\Sigma$, instead of inclusion,
to get a ``minimal'' cubing.
Their results assumed the ambient group $G$, as well as the $H_j$'s, to be finitely generated.
In this paper, I remove all the finite generation assumptions and instead assume
that all the $X_j$'s come from splittings collectively satisfying sandwiching (see Section~\ref{VeryGoodPosition}).
The main challenge in adapting the Cayley graph arguments from~\cite{NibloSageevScottSwarup2005}
to the non-finitely generated case is to show that the cubing is nonempty.
In particular:
\begin{itemize}
	\item
	Constructing an ultrafilter on $(\Sigma, \leq)$ (see the first half of proof of Theorem~\ref{ULTRAFILTER}), and
	\item
	Proving that this ultrafilter satisfies the descending chain condition (see the second half of proof of Theorem~\ref{ULTRAFILTER}).
\end{itemize}
One application of minimal cubings is putting the $X_j$'s in ``very good position,'' i.e. replacing each $X_j$ by some $X'_j$,
such that on the set of all translates of the $X'_j$'s and their complements,
inclusion is the same as almost inclusion.

In~\cite{ScottSwarup2003}, Scott and Swarup defined
the algebraic regular neighborhood of a finite collection of almost invariant subsets of a finitely generated group $G$,
and proved the existence and uniqueness of algebraic regular neighborhoods of almost invariant sets.
Scott and Swarup assumed that the minimal cubing from~\cite{NibloSageevScottSwarup2005} could be turned into an algebraic regular neighborhood of the $X_j$'s.
However, they ran into trouble proving that the object they constructed satisfied the definition of an algebraic regular neighborhood,
namely because edge stabilizers might not be finitely generated.
I include the missing arguments in Section~\ref{VGPforGFG}.

In Section~\ref{AlgebraicRegularNeighborhoodsE},
I use Scott and Swarup's suggestion of how to turn a minimal cubing into an algebraic regular neighborhood under my hypotheses
(i.e. the almost invariant sets come from splittings satisfying sandwiching, and no groups need to be finitely generated),
and I successfully prove that object constructed does indeed satisfy the definition of an algebraic regular neighborhood.
In Section~\ref{AlgebraicRegularNeighborhoodsU}, I prove that algebraic regular neighborhoods are unique,
even for a possibly infinite collection of splittings.

\subsection*{Acknowledgements} 
Partially supported by NSF RTG grant 0602191.

Thanks to Peter Scott for posing many of the questions answered in this paper, and for many productive discussions.

\section{Preliminaries and Main Ideas}
\label{Preliminaries}

The purpose of this section is to:
\begin{enumerate}
	\item
	Give enough background information
	to enable to reader to understand the statements of all the results in this paper, and
	\item
	Provide an idea of how everything fits together.
\end{enumerate}

\subsection{Splittings}

A {\bf $G$-tree} is a simplicial tree equipped with a (simplicial) $G$-action,
such that the action does not invert any edges.
A {\bf splitting} of a group $G$ is a $G$-tree $T$ with no global fixed points,
such that the quotient graph $G \setminus T$ has exactly one edge.
There are two cases:

\begin{enumerate}
	\item
	$G \setminus T$ consists of one edge with distinct endpoints.
	Pick an edge $e$ of $T$, let $H$ denote the stabilizer of $e$, and let $A$ and $B$ denote the stabilizers of the endpoints of $e$.
	We have inclusions $i_1 : H \hookrightarrow A$ and $i_2 : H \hookrightarrow B$.
	We call $\sigma$ an {\bf amalgamated free product} and may write {\bf $\sigma : G \cong A *_H B$}.
	Note that using Bass-Serre theory (see \cite{Serre1980}),
	we could reconstruct $T$ using only the inclusions $i_1 : H \hookrightarrow A$ and $i_2 : H \hookrightarrow B$.
	One presentation for $G$ is $\langle A, B | i_1 (h) = i_2 (h), \text{ for all } h \in H \rangle$.
	\item
	$G \setminus T$ consists of a loop with one edge and one vertex.
	Pick and edge $e$ of $T$, let $H$ denote the stabilizer of $e$, and let $A$ denote the stabilizer of one of the endpoints of $e$.
	We have an inclusion $i_1 : H \hookrightarrow A$.
	As there is only one orbit of vertices, the other endpoint of $e$ is
	a translate of the first endpoint by some $t \in G$.
	The stabilizer of this vertex is $t A t^{-1}$, so we have an inclusion $H \hookrightarrow t A t^{-1} \cong A$.
	We can view this second inclusion as $i_2 : H \hookrightarrow A$,
	where $i_2(h) := t^{-1} i_1(h) t$.
	We call $\sigma$ an {\bf HNN extension} and may write {\bf $\sigma : G \cong A*_H$}.
	Note that using Bass-Serre theory (see \cite{Serre1980}),
	we could recover $T$ using only the two inclusions $i_1$,~$i_2 : H \hookrightarrow A$.
	One presentation for $G$ is $\langle A, t | i_2 (h)~=~t^{-1} i_1 (h) t, \text{ for all } h~\in~H \rangle$.
\end{enumerate}
In either case, we call $\sigma$ a {\bf splitting of $G$ over $H$}. The subgroup $H$ is well-defined up to conjugacy in $G$. We call two splittings of a group {\bf isomorphic} if there exists a $G$-equivariant isomorphism between the trees for the two splittings.
In most of this paper, existence results (in particular, Theorems~\ref{COMPATIBLE},~\ref{VGP}, and~\ref{thmARNE})
only work for a finite collection of splittings,
while uniqueness results do not require such an assumption.

If the $G$-tree for a splitting is a line on which $G$ acts by translations only,
we call the splitting a
\label{TRIVASC}
{\bf trivially ascending HNN extension}.
Note that each edge in the line has the same stabilizer, denoted $H$,
so that $H$ acts trivially on the line.
Equivalently, $H$ is normal in $G$ and $\frac{G}{H} \cong \mathbb{Z}$.
Equivalently, the splitting has the form $A *_H$ where
both inclusions $i_1 : H \hookrightarrow A$ and $i_2: H \hookrightarrow A$ are isomorphisms.

To describe how two splittings of $G$ cross, for each splitting we will construct a subset of $G$, and then look at how the two subsets and their translates cross.
We show how to construct the subsets in Section~\ref{SUBaifromsplitting}.

\subsection{Almost Invariant Sets and Crossing}

Many concepts used here consider subgroups of $G$
``up to finite index'' and subsets of $G$ ``up to finitely many cosets.''
Here are a few key definitions capturing this idea.
\begin{defn}
Let $H$ and $K$ be subgroups of a group $G$.
\begin{itemize}
	\item
	We say $H$ and $K$ are {\bf commensurable}
	if $H \cap K$ has finite index in $H$ and in $K$.
	\item
	A subset of $G$ is {\bf $H$-finite}
	if it is contained in finitely many cosets $H g_i$ of $H$ in $G$.
	\item
	Two subsets $A$ and $B$ of $G$ are {\bf $H$-almost equal}, written $A \overset{H-a}{=} B$,
	if their symmetric difference is $H$-finite.		
\end{itemize}
\end{defn}

An almost invariant subset of a group is a subset which does not change by much
when you multiply on the right by an element of the group.
Specifically:

\begin{defn}
Let $G$ be any group, $H$ and subgroup of $G$, and $X$ any subset of $G$.
We say $X$ is an {\bf $H$-almost invariant} subset of $G$
if the following two properties are satisfied:
\begin{enumerate}
	\item
	$H$ stabilizes $X$, i.e. $hX = X$, for all $h \in H$, and
	\item
	$Xg \overset{H-a}{=}   X$, i.e. the symmetric difference of $Xg$ and $X$ is $H$-finite, for all $g \in G$.
\end{enumerate}
\end{defn}

Let $X$ be $H$-almost invariant and $Y$ be $K$-almost invariant.
We call the four sets $X \cap Y$, $X \cap Y^*$, $X^* \cap Y$, and $X^* \cap Y^*$
the {\bf corners of the pair $(X,Y)$}.
\begin{defn}
\label{DEFNcrossing}
	Let $X$ and $Y$ be subsets of $G$.
	$X$ and $Y$ are {\bf nested} if $X$ or $X^*$ is a subset of $Y$ or $Y^*$,
	i.e. a corner of the pair $(X,Y)$ is empty.
	Otherwise, $X$ and $Y$ are {\bf not nested}.
\end{defn}
We'd like a similar notion that works ``up to finitely many cosets.''
\begin{defn}
	Let $X$ be an $H$-almost invariant subset of $G$,
	and $Y$ a $K$-almost invariant subset of $G$.
	The pair $(X,Y)$ is {\bf almost nested} if
	a corner of the pair $(X,Y)$ is $K$-finite.
	\\
	Otherwise, {\bf $X$ crosses $Y$}, i.e.
	no corner of the pair $(X,Y)$ is $K$-finite.
\end{defn}
A couple of facts justify this definition:
\begin{enumerate}
	\item
	If $X$ and $Y$ arise from splittings of $G$,
	or if $X$ and $Y$ do not necessarily come from splittings but $G$ is finitely generated,
	then $X$ crosses $Y$ if, and only if, $Y$ crosses $X$.
	For the proof in the case where $X$ and $Y$ come from splittings, see Proposition~\ref{SYMMETRIC}.
	For the proof in the case where $G$ is finitely generated, see Lemma 2.3 of \cite{Scott1998}.	
	\item
	If $Y$ is both $K$-almost invariant and $K'$-almost invariant, then $K$ and $K'$ must be commensurable
	(see Lemma~\ref{COMMENSURABLE});
	in particular, $K$-finiteness is the same as $K'$-finiteness.
	Also note that if $Y$ is $K$-almost invariant, then $K \subset Stab(Y)$,
	so that $K$ must be a finite-index subgroup of $Stab(Y)$.
\end{enumerate}

Now we are ready to define the intersection number of two splittings $\sigma$ and $\tau$ of $G$.
Let $X$ be an $H$-almost invariant set arising from $\sigma$, and $Y$ a $K$-almost invariant set arising from $\tau$.
To compute the intersection number of $X$ and $Y$, count the number of $g \in G$ such that $gX$~crosses~$Y$,
then eliminate double-counting.
If $h \in H$ and $k \in K$, then $hX = X$ and $kY = Y$,
so $gX$ crossing $Y$ is the same as $k^{-1} g h X$ crossing $Y$.
Define the
{\bf intersection number of $\sigma$ and $\tau$}
by:
\label{DEFNintersectionnumber}
\[
	i(\sigma, \tau) := \text{number of double cosets } KgH \text{ such that } gX \text{ crosses } Y.
\]
In this paper, we will mostly only care whether the intersection number of two splittings is non-zero,
i.e. whether any translates of $X$ and $Y$ cross each other.

\subsection{Almost Invariant Sets Arising from Splittings}
\label{SUBaifromsplitting}

Let $\sigma$ be a splitting of $G$, and let $T$ be a $G$-tree for $\sigma$.
Pick a base vertex $v$ and a (directed) edge $e$ of $T$. We define a subset $X$ of $G$ by:
\[
	X := \{g \in G | e \text{ points away from } gv\}.
\]
Let $H$ denote the stabilizer of $e$.
Such $X$ is in fact an $H$-almost invariant subset of $G$;
clearly $H$ stabilizes $X$, and
for a proof of the second property, see Lemma~\ref{BASICEQUIV}.

Some almost invariant sets arise from splittings, and others do not.
Note that if $X$ is a standard almost invariant set arising from $\sigma$,
then so is each translate of $X$ and its complement.

We will often go back and forth between a splitting $\sigma$ and an almost invariant set $X$ arising from $\sigma$.
Given a splitting $\sigma$, we can construct $X$ by picking a base vertex and edge in the tree for $\sigma$.
Given an almost invariant set $X$ arising from a splitting $\sigma$,
let $\Sigma$ denote the set of all translates of $X$ and its complement, partially ordered by inclusion.
We can apply Dunwoody's theorem (see Section~\ref{SUBdunwoody}) to produce a $G$-tree.
This tree will yield a splitting isomorphic to $\sigma$ (see Proposition~\ref{EQUIVALENT} for a proof).
The choice of base edge $e$ and base vertex $v$ is necessary to define $X$, but it does not particularly matter which one we choose.
A change in $e$ will result in $X$ being replaced by a translate of $X$ or its complement
(and hence does not change the set $\Sigma$).
A change in $v$ will result in an almost invariant set $X' \overset{H-a}{=} X$ (see Lemma~\ref{BASICEQUIV}),
and hence yields a splitting isomorphic to $\sigma$ (see Proposition~\ref{EQUIVALENT}).

We will use the following convention: ``$X$ is an $H$-almost invariant set arising from the splitting $\sigma$''
implicitly means that $H$ is equal to the stabilizer of $X$ (and not a proper subgroup of $Stab(X)$), and $\sigma$ is a splitting of $G$ over $H$.

\subsection{Dunwoody's Theorem}
\label{SUBdunwoody}

Dunwoody's Theorem takes in a partially ordered set satisfying tree-like properties,
and spits out a tree.
Some applications include
Dunwoody's proof of Stallings' theorem in \cite{Dunwoody1979} (also see~\cite{ScottWall1979}),
reconstructing the Bass-Serre tree for a splitting
by using an almost invariant set that came from the splitting (see Section~\ref{SUBaifromsplitting}),
and constructing a common refinement for trees representing two splittings that have intersection number zero
(see Section~\ref{SUBintzero}).

Take any simplicial tree and let $\Sigma$ denote its (directed) edge set.
Reversing the direction of an edge gives us a free involution $*$ on $\Sigma$,
and we can describe an undirected edge as a pair $\{e, e^*\}$.
We describe some tree-like properties that $\Sigma$ satisfies:
\begin{itemize}
	\item
	If there is an edge path starting with the edge $e$ and ending with $f$,
	then there is an edge path starting with $f^*$ and ending with $e^*$.
	\item
	For any two undirected edges,
	we can find a simple edge path connecting them.
	\item
	$\Sigma$ has no loops.
\end{itemize}
One can define a partial order on $\Sigma$ as follows:
\[
	e \leq f \iff \text{ there exists a simple edge path starting at } e \text{ and ending at } f.
\]

Dunwoody's theorem states that any partially ordered set $(\Sigma, \leq)$
satisfying analogous properties
can be turned into the edge set of a tree.
\begin{thm}[DunwoodyÕs Theorem]
	Let $(\Sigma, \leq)$ be a partially ordered set equipped with a free involution $*$ on $\Sigma$.
	Suppose the following conditions are satisfied:
	\begin{enumerate}
		\item
		For all $A, B \in \Sigma$, if $A \leq B$, then $B^* \leq A^*$.
		\item
		For all $A, B \in \Sigma$ with $A \leq B$,
		there are only finitely many $C \in \Sigma$ with $A \leq C \leq B$.
		\item
		For all $A, B \in \Sigma$, at least one of the four relations $A^{(*)} \leq B^{(*)}$ holds.
		\item
		We cannot have simultaneously $A \leq B$ and $A \leq B^*$.
	\end{enumerate}
	Then there exists a tree $T$ with (directed) edge set $\Sigma$, and such that
	$A \leq B$ if, and only if, there exists a simple edge path whose first edge is $A$ and whose last edge is $B$.
\end{thm}

The key idea in the proof of Dunwoody's theorem is constructing the vertices of $T$.
Let each element of $\Sigma$ be a directed edge,
and make a vertex wherever you have two edges with nothing in between.
Specifically, define the vertices of the tree to be equivalence classes of elements of~$\Sigma$:
\begin{equation*}
	[e] = [f] \iff
	\begin{array}{ll}
		e \leq f^*, \text{ AND}\\
		\text{if } e \leq a \leq f^*, \text{ then } a = e \text{ or } a = f^*.\\
	\end{array}
\end{equation*}
Then one must prove that everything works out.

\subsection{Almost Inclusion and Small Corners}
\label{SUBpo}

Given an $H$-almost set $X$, we will often want to refer to the set of all translates of $X$ and its complement.
Denote this set by $\Sigma(X)$.
\[
	\Sigma(X) := \{gX, gX^* | g \in G\}.
\]

For the remainder of the paragraph, let $X$ be an $H$-almost invariant subset of $G$,
let $Y$ be a $K$-almost invariant subset of $G$,
and let $\Sigma$ denote $\Sigma(X) \cup \Sigma(Y)$.
Also assume that we are in a situation where crossing is symmetric - 
for example, assume $G$ is finitely generated
or assume that $X$ and $Y$ come from splittings.
Given $A,B \in \Sigma$, 
we say a corner of the pair $(A,B)$ is {\bf small} if it is
contained in finitely many right $Stab(A)$-cosets.
By the above remarks, this is equivalent to the corner being $Stab(B)$-finite.
In this paper, we'll only use the term ``small'' when we know that crossing is symmetric.

Inclusion partially orders $\Sigma$;
however, we'd prefer a partial order that isn't affected
when you change $X$ by finitely many $H$-cosets or $Y$ by finitely many $K$-cosets.
The obvious thing to do is declare $A \leq B$ precisely when $A \cap B^*$ is small.
However, we run into a potential difficulty: if two corners of a given pair $(A,B)$ are small,
how do we decide which inequality to choose?
If two corners are small and one of them is empty,
then choose to only pay attention to the empty corner.
For example, if $A \cap B^*$ and $A^* \cap B$ are small, and $A^* \cap B$ is empty (i.e. $B \subset A$),
then we declare $B \leq A$.
We say that $\Sigma$ is in {\bf good position}
if for all $A,B \in \Sigma$,
if two corners of the pair $(A,B)$ are small, then one is empty.
We can define a relation $\leq$ on $\Sigma$ by:
\label{DEFNPO}
\[
	A \leq B \iff A \cap B^* \text{ is empty, or is the only small corner of the pair } (A,B).
\]
It turns out that if $\Sigma$ is in good position,
then $\leq$ is a partial order on $\Sigma$ (see Corollary~\ref{PARTIALORDER} for a proof).
I show that good position is automatic if $X$ and $Y$ arise from non-isomorphic splittings
(see Corollary~\ref{GOODPOSITION}).

\subsection{Example: Simple Closed Curves on a Surface}

To gain more intuition about splittings, we look at a few concrete examples.
Let $S$ be a closed, orientable surface of genus at least two.
Let $G$ denote the fundamental group of $S$.
Let $\gamma$ be a $\pi_1$-injective simple closed curve on $S$.
Let $\pi: \widetilde{S} \rightarrow S$ denote the universal cover of $S$.
The preimage $\pi^{-1}(\gamma)$ is a collection of disjoint lines.
Pick one of these lines, and call it $l$. Let $H$ denote the stabilizer of $l$, so that $H \cong \mathbb{Z}$.
We can construct a tree $T$ whose directed edge set is $\Sigma(X)$;
simply take the dual tree to $\pi^{-1}(\gamma)$ in $\widetilde{S}$.
After choosing basepoints, $G$ acts on $T$ via deck transformations, with no fixed points or edge inversions.
The stabilizer of an edge is isomorphic to $\mathbb{Z}$, so we have a splitting $\sigma$ of $G$ over $\mathbb{Z}$.
If $\gamma$ separates $S$ into two components $S'$ and $S''$,
then $\sigma$ is an amalgamated free product $\pi_1(S') *_\mathbb{Z} \pi_1 (S'')$ (see Figure~\ref{fig:Gtwo-1}).
If $\gamma$ does not separate $S$, then $\sigma$ is an HNN extension $\pi_1(S - \gamma) *_\mathbb{Z}$ (see Figure~\ref{fig:Gtwo-2}).

\begin{figure}[!htbp]
	\centering
	\begin{minipage}[h]{0.48\linewidth}
		\includegraphics[width=65mm]{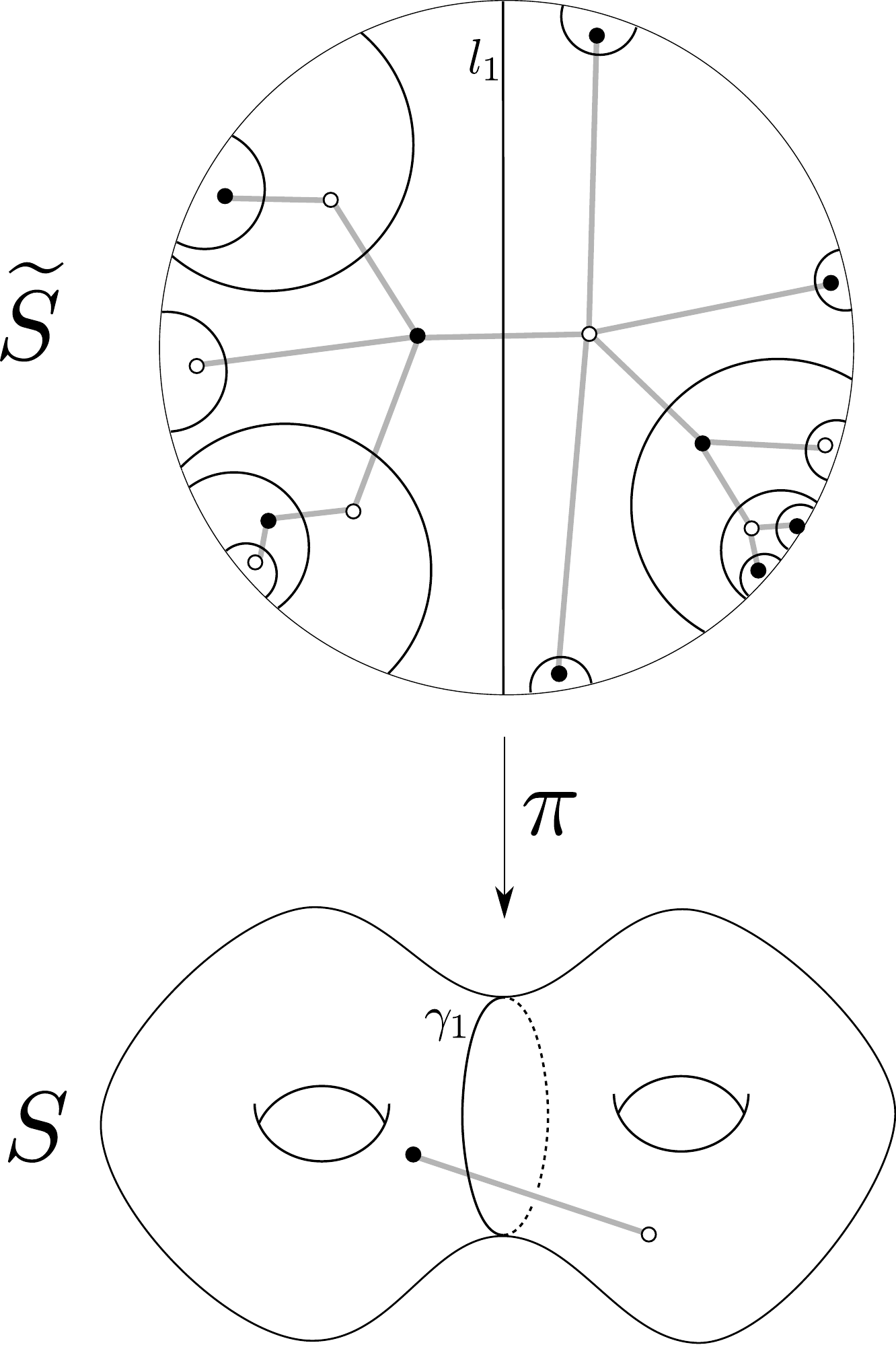}
	\end{minipage}
	\hfill
	\begin{minipage}[h]{0.48\linewidth}
		\caption[Separating loop]{
		\label{fig:Gtwo-1}
		Schematic picture: $\pi^{-1}(\gamma_1)$, the union of $l_1$ and all its translates,
		is a collection if disjoint lines.
		The complement $\widetilde{S} - \pi^{-1}(\gamma_1)$ has two types of components: those that project to the left of $\gamma_1$, and those that project to the right of $\gamma_1$.
		Correspondingly, the dual tree to $\pi^{-1}(\gamma_1)$ has two orbits of vertices.
		The action of the fundamental group of $S$ on the tree gives an amalgamated free product.
		}
	\end{minipage}
\end{figure}

\begin{figure}[!htbp]
	\centering
	\begin{minipage}[h]{0.48\linewidth}
		\includegraphics[width=65mm]{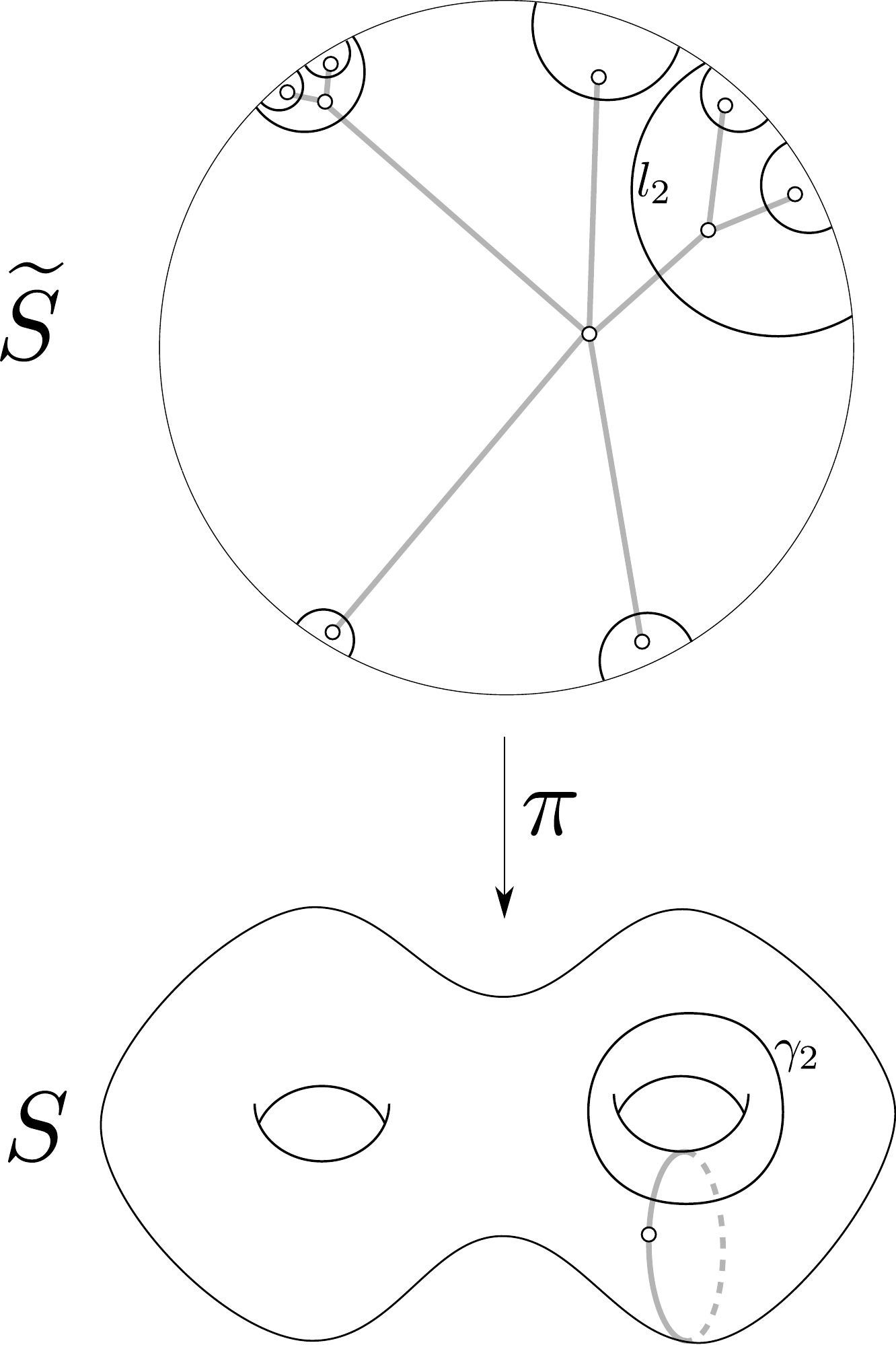}
	\end{minipage}
	\hfill
	\begin{minipage}[h]{0.48\linewidth}
		\caption[Non-separating loop]{
		\label{fig:Gtwo-2}
		Schematic picture: $\pi^{-1}(\gamma_2)$, the union of $l_2$ and all its translates,
		is a collection if disjoint lines.
		$S - \gamma$ has only one region, so the dual tree to $\pi^{-1}(\gamma_2)$
		has only one orbit of vertices.
		The action of the fundamental group of $S$ on the tree gives an HNN extension.
		}
	\end{minipage}
\end{figure}

Next, consider the curves $\gamma_1$ and $\gamma_2$ from Figures~\ref{fig:Gtwo-1} and~\ref{fig:Gtwo-2} simultaneously,
as in Figure~\ref{fig:Gtwo-12}.
We have associated splittings $\sigma_1$ and $\sigma_2$ of $\pi_1(S)$.
Since $\gamma_1$ and $\gamma_2$ do not cross each other,
the dual graph $T$ to $\pi^{-1}(\gamma_1 \cup \gamma_2)$ is a tree.
Moreover, $T$ is a common refinement of the dual tree to $\pi^{-1}(\gamma_1)$
and the dual tree to $\pi^{-1}(\gamma_2)$.
Hence trees for $\sigma_1$ and $\sigma_2$ have a common refinement.
This is an example of compatible splittings (see Section~\ref{SUBintzero}).

\begin{figure}[!htbp]
	\centering
	\begin{minipage}[h]{0.48\linewidth}
		\includegraphics[width=65mm]{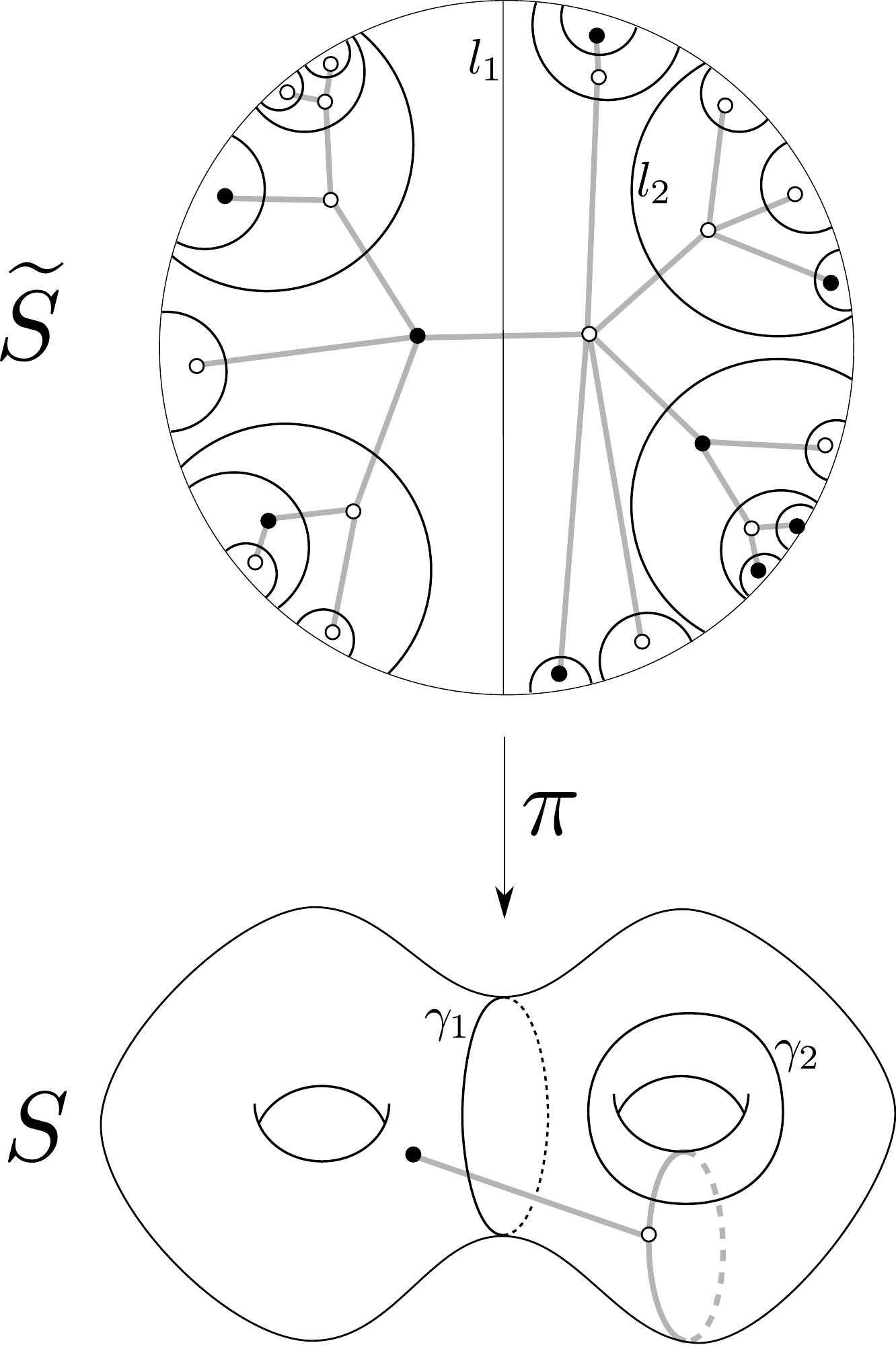}
	\end{minipage}
	\hfill
	\begin{minipage}[h]{0.48\linewidth}
		\caption[Two loops that do not intersect]{
		\label{fig:Gtwo-12}
		Schematic picture: $\pi^{-1}(\gamma_1 \cup \gamma_2)$,
		the union of $l_1$, $l_2$, and all their translates, is a collection of distinct lines,
		each projecting to either $\gamma_1$ or $\gamma_2$.
		Correspondingly, the dual tree to $\pi^{-1}(\gamma_1 \cup \gamma_2)$ has two orbits of edges.
		The complement $S - \pi^{-1}(\gamma_1 \cup \gamma_2)$ has two types of components:
		those that project to the left of $\gamma_1$, and those that project to the right of $\gamma_1$.
		Correspondingly, the dual tree to $\pi^{-1}(\gamma_1 \cup \gamma_2)$ has two orbits of vertices.
		$\gamma_1$ and $\gamma_2$ yield two splittings of the fundamental group of $S$,
		and $T$ is a compatibility tree for the splittings.
		}
	\end{minipage}
\end{figure}

What if we ``poke a finger'' out of $\gamma_2$, as in Figure~\ref{fig:Gtwo-12poke}?
This gives the same two splittings $\sigma_1$ and $\sigma_2$ of $\pi_1(S)$ as in the previous paragraph.
However, from the way $\gamma_1$ and $\gamma_2$ are drawn,
the dual graph to $\pi^{-1}(\gamma_1 \cup \gamma_2)$ is no longer a tree.
In order to find a common refinement of the trees for $\sigma_1$ and $\sigma_2$,
it is helpful to first pull $\gamma_1$ and $\gamma_2$ tight to geodesics.
When dealing with arbitrary splittings (not just those induced by simple closed curves on surfaces),
we'll need some sort of algebraic tool choose nice representatives for splittings.

\begin{figure}[!htbp]
	\centering
	\begin{minipage}[h]{0.48\linewidth}
		\includegraphics[width=65mm]{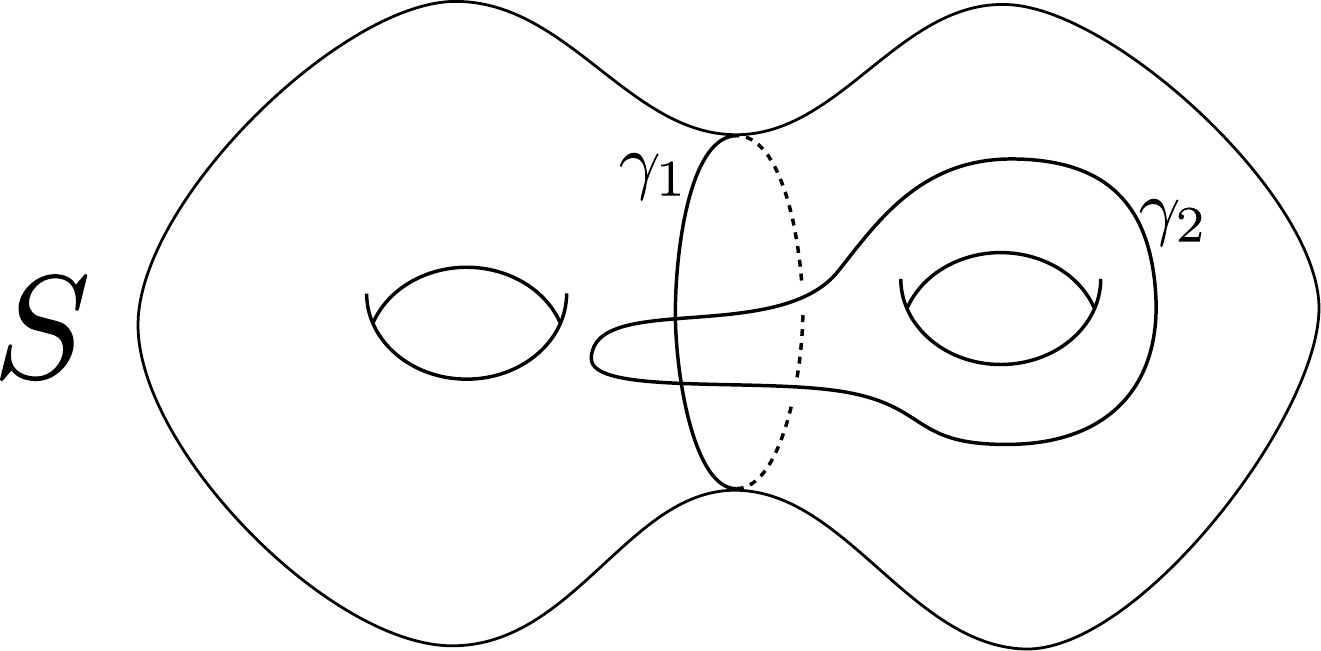}
	\end{minipage}
	\hfill
	\begin{minipage}[h]{0.48\linewidth}
		\caption[Two loops that intersect trivially]{
		\label{fig:Gtwo-12poke}
		Take the previous example,
		but deform $\gamma_2$ slightly.
		The splittings induced by $\gamma_1$ and $\gamma_2$ still have intersection number zero,
		but in this example,
		the dual graph to $\pi^{-1}(\gamma_1 \cup \gamma_2)$ is not a tree.
		Since $\gamma_1$ and the new $\gamma_2$ have ``inessential crossing,''
		the are not the best curves to use for the splittings.
		}
	\end{minipage}
\end{figure}

\subsection{Sandwiching}

At times we will need to assume that either $X$ crosses all translates of $Y$, or $X$ can be sandwiched between two translates of $Y$ or $Y^*$.
\begin{defn}[Modified from \cite{ScottSwarup2003errata}]
\label{DEFNSANDWICHING}
Let $\{X_j | j \in J\}$ denote any collection of almost invariant subsets of $G$,
where $\mycup \limits_{j \in J} \Sigma(X_j)$ is partially ordered by $\leq$.
\begin{itemize}
	\item
	{\bf $X_j$ is sandwiched by $X_k$}
	if either there exist $A, B \in \Sigma(X_k)$ such that $A \leq X_j \leq B$, \\
	or $X_j$ crosses every element of $\Sigma(X_k)$.
	\item
	{\bf $\{X_j | j \in J\}$ satisfies sandwiching}
	if for all $j,k \in J$,
	we have $X_j$ is sandwiched by $X_k$.
	\item
	If $X_j$ is a standard almost invariant set arising from a splitting $\sigma_j$,\\
	then we say {\bf $\{\sigma_j | j \in J\}$ satisfies sandwiching}.
	Note that by Lemma~\ref{BASICEQUIV}, it does not matter which $X_j$ we choose to represent $\sigma_j$.
\end{itemize}
\end{defn}
Most collections of splittings satisfy sandwiching.
In fact, if none of the $X_j$'s yields a trivially ascending HNN extension,
then $\{X_j | j \in J\}$ automatically satisfies sandwiching
(see Corollary~\ref{TRIVHNN}).

\subsection{Compatibility}
\label{SUBintzero}

We call two splittings ``compatible'' if their trees have a common refinement.
Here is the formal definition, which works for an arbitrary number of splittings:
\begin{defn}
Let $\{\sigma_j | j \in J\}$ be any collection of splittings of $G$.
A {\bf compatibility tree for $\{\sigma_j | j \in J\}$}
is a $G$-tree $T$ consisting of one edge orbit for each $j$,
such that collapsing all edges except the $\sigma_j$-edges yields a tree for $\sigma_j$.
We say {\bf $\{\sigma_j | j \in J\}$ is compatible}
if $\{\sigma_j | j \in J\}$ has a compatibility tree.
\end{defn}
Note that in the definition of ``compatibility tree,'' we insist on having a bijection
between $J$ and the set of edge orbits.

\begin{exmp}
Any splitting is compatible with (any splitting isomorphic to) itself.
To see this, take a $G$-tree for the splitting, and subdivide each edge in two.
\end{exmp}

Since it is not possible to have a $G$-tree with infinitely many distinct edge orbits yielding isomorphic splittings,
if $\{\sigma_j | j \in J\}$ is compatible then only finitely many of the $\sigma_j$'s can belong to any given isomorphism class.
Compatibility trees are always unique (see Corollary~\ref{UNIQUENESSTREE}),
even for infinite collections of splittings.

I prove that if $i(\sigma, \tau)=0$ and $\sigma$ and $\tau$ satisfy sandwiching,
then $\sigma$ and $\tau$ are compatible
(see Theorem~\ref{COMPATIBLE}).
The main idea of the proof is to note that if $i(\sigma, \tau)=0$,
then for all $A, B \in \Sigma$, we have $A$ and $B$ are almost nested, i.e. one of $A \leq B, A \leq B^*, A^* \leq B$, or $A^* \leq B^*$.
Then apply Dunwoody's theorem (see Section~\ref{SUBdunwoody}).

Note that if two splittings are compatible,
then we can find corresponding almost invariant sets that are nested (instead of just almost nested) as follows.
Suppose that $X$ is an $H$-almost invariant set arising from $\sigma$,
that $Y$ is a $K$-almost invariant set arising from $\tau$,
and that $T$ is a compatibility tree for $\sigma$ and $\tau$.
Take $\leq$, as defined in Section~\ref{SUBpo}.
Since $i(\sigma,\tau) = 0$, we have for all $A, B \in \Sigma$, one of $A^{(*)} \leq B^{(*)}$.
We apply Lemma~\ref{BASICEQUIV}  to $T$ to get
$X' \overset{H-a}{=} X$ and $Y' \overset{H-a}{=} Y$,
such that for all $A', B' \in \Sigma(X') \cup \Sigma(Y')$, one of ${A'}^{(*)} \subset {B'}^{(*)}$.

\subsection{Sandwiching is Necessary}
\label{NECESSARY}

If a $G$-tree has two edge orbits and no fixed points,
then for any given edge, we can find two edges on either side of it belonging to the other edge orbit.
This shoes that if two splittings are compatible, then the splittings necessarily satisfy sandwiching.

Guirardel gave an example of two splittings with intersection number zero that do not satisfy sandwiching
(and hence are not compatible). These are splittings of the free group on two generators, and over non-finitely generated subgroups.
One of the splittings is a trivially ascending HNN extension.
See \cite{ScottSwarup2003errata} for the construction.

\subsection{Turning Almost Inclusion Into Inclusion}

What happens if we try to apply ideas from the ``intersection number zero implies compatible'' theorem to splittings having positive intersection number?

Let $\sigma_1, \sigma_2, \ldots, \sigma_n$ be splittings of $G$ collectively satisfying sandwiching,
and such that no two of the splittings are isomorphic.
Let $X_j$ be an almost invariant set arising from $\sigma_j$.
As on Page~\pageref{DEFNPO}, we can define a partial order $\leq$ on $\mycup \limits_{j=1}^{n} \Sigma(X_j)$, such that if $A$ is a subset of $B$, then $A \leq B$.
It turns out that we can build a $CAT(0)$ cubical complex,
then use Theorem~\ref{VGP} to replace each $X_j$ by another almost invariant set $X'_j \overset{H_j-a}{=} X_j$
such that for all $A, B \in \mycup \limits_{j=1}^{n} \Sigma(X'_j)$,
we have $A \subset B$ if, and only if, $A \leq B$. This is called putting the $X_j$'s in
{\bf very good position}.
Details for this construction are laid out in Section~\ref{VeryGoodPosition}.
The same $CAT(0)$ cubical complex can also be used to construct an algebraic regular neighborhood of the $\sigma_j$'s
(see Section~\ref{AlgebraicRegularNeighborhoodsE}).

\subsection{Algebraic Regular Neighborhoods}

The notion of ``algebraic regular neighborhood'' is a generalization of PL regular neighborhood, up to homotopy.
Let $\gamma_1, \ldots, \gamma_n$ be $\pi_1$-injective simple closed curves on a surface $S$,
yielding splittings $\sigma_1, \ldots, \sigma_n$ of $\pi_1(S)$.
Let $N$ be a regular neighborhood of the $\gamma_j$'s.
Assume each component of the frontier of $N$ is $\pi_1$-injective.
We will construct a bipartite graph dual to the frontier of $N$.
For each component of $N$, add a $V_0$-vertex.
For each component of $S - N$, add a $V_1$-vertex.
For each component of the frontier of $N$, add an edge connecting the corresponding $V_0$- and $V_1$-vertices.
The preimage of $\Gamma$ in $\widetilde{S}$ is a tree.
We call this tree an algebraic regular neighborhood of $\{\sigma_1, \ldots, \sigma_n\}$.
Each $V_0$-vertex orbit encloses (to be defined below) some of the $\sigma_j$'s.
Each simple closed curve disjoint from all the $\gamma_j$ can be homotoped to be disjoint from $N$,
and hence is enclosed in a $V_1$-vertex.
See Figure~\ref{fig:Gtwo-123} for a concrete example.

For a beautiful explanation of the relationship between algebraic regular neighborhoods
and JSJ decompositions of $3$-manifolds,
see the Introduction in \cite{ScottSwarup2003}.

In \cite{ScottSwarup2003}, the authors defined an algebraic regular neighborhood of a family of almost invariant sets as a graph of groups,
and they defined what it means for a vertex to enclose an almost invariant set.
Here, to avoid confusion about base points, we define algebraic regular neighborhood as a $G$-tree.
Also, since we're only working with almost invariant sets arising from splittings,
we define what it means for the orbit of a vertex of a $G$-tree to enclose a splitting.

\begin{defn}
Let $T$ be a $G$-tree, $V$ a vertex of $T$, and $\sigma$ a splitting of $G$.
We say {\bf the orbit of $V$ encloses $\sigma$}
if we can refine $T$ by inserting an edge at each vertex in the orbit of $V$,
such that the new edges form a tree for $\sigma$.
More precisely, there exists a $G$-tree $T'$
and a $G$-orbit of edges in $T'$ (call these edges {\bf $\sigma$ edges})
such that both of the following hold:
\begin{enumerate}
	\item
	There exists a $G$-equivariant isomorphism
	\[
		T' / \text{ all non-} \sigma \text{ edges collapsed} \cong \text{ tree for } \sigma.
	\]
	\item
	There exists a $G$-equivariant isomorphism
	\[
		T' / \text{ all } \sigma \text{ edges collapsed} \cong T.
	\]
\end{enumerate}
\end{defn}

Given a collection $\{\sigma_j | j \in J\}$ of splittings, we call $\sigma_j$ an {\bf isolated splitting}
if it has intersection number zero with each other splitting $\sigma_k$.
Given a $G$-tree, we call a vertex {\bf isolated} if the vertex has valence two in both the $G$-tree and the tree's quotient under the $G$-action.
In an algebraic regular neighborhood,
we'll want each isolated splitting in $\{\sigma_j | j \in J\}$
to be enclosed by the orbit of an isolated $V_0$-vertex.
Now we are ready to formally define an algebraic regular neighborhood.
\begin{defn}[Reformulated from \cite{ScottSwarup2003}]
\label{defnARN}
Let $G$ be any group with any collection $\{\sigma_j | j \in J\}$ of pairwise non-isomorphic splittings.
Suppose $\{\sigma_j | j \in J\}$ satisfies sandwiching.
\label{DEFNofARN}
	An {\bf algebraic regular neighborhood of $\{\sigma_j | j \in J\}$}
	is a bipartite $G$-tree $T$
	(denote the two vertex colors by $V_0$ and $V_1$)
	satisfying the following five conditions:
	\begin{enumerate}
		\item
		Each $\sigma_j$ is enclosed by some $V_0$-vertex orbit in $T$,
		and each $V_0$-vertex orbit encloses some~$\sigma_j$.
		\item
		If $\sigma$ is a splitting of $G$ over $H$,
		where for all $j \in J$
		$\sigma$ is sandwiched by $\sigma_j$
		and $i(\sigma, \sigma_j) = 0$,
		then $\sigma$ is enclosed by some $V_1$-vertex orbit in $T$.
		\item
		$T$ is a minimal $G$-tree.
		\item
		There exists a bijection
		\[
			f: 
			\{j \in J | \sigma_j \text{ is isolated}\}
			{\rightarrow} G\text{-orbits of isolated } V_0 \text{-vertices of } T
		\]
		such that $f(j)$ encloses $\sigma_j$.
		\item
		Every non-isolated $V_0$-vertex orbit in $T$ encloses some non-isolated $\sigma_j$.
	\end{enumerate}
\end{defn}

In Section~\ref{AlgebraicRegularNeighborhoodsE}, I prove the existence of algebraic regular neighborhoods
for any finite collection of splittings satisfying sandwiching.
In Section~\ref{AlgebraicRegularNeighborhoodsU}, I prove uniqueness of algebraic regular neighborhoods
for possibly infinite collections of splittings satisfying sandwiching.

\begin{figure}[!htbp]
	\centering
	\begin{minipage}[h]{0.48\linewidth}
		\includegraphics[width=50mm]{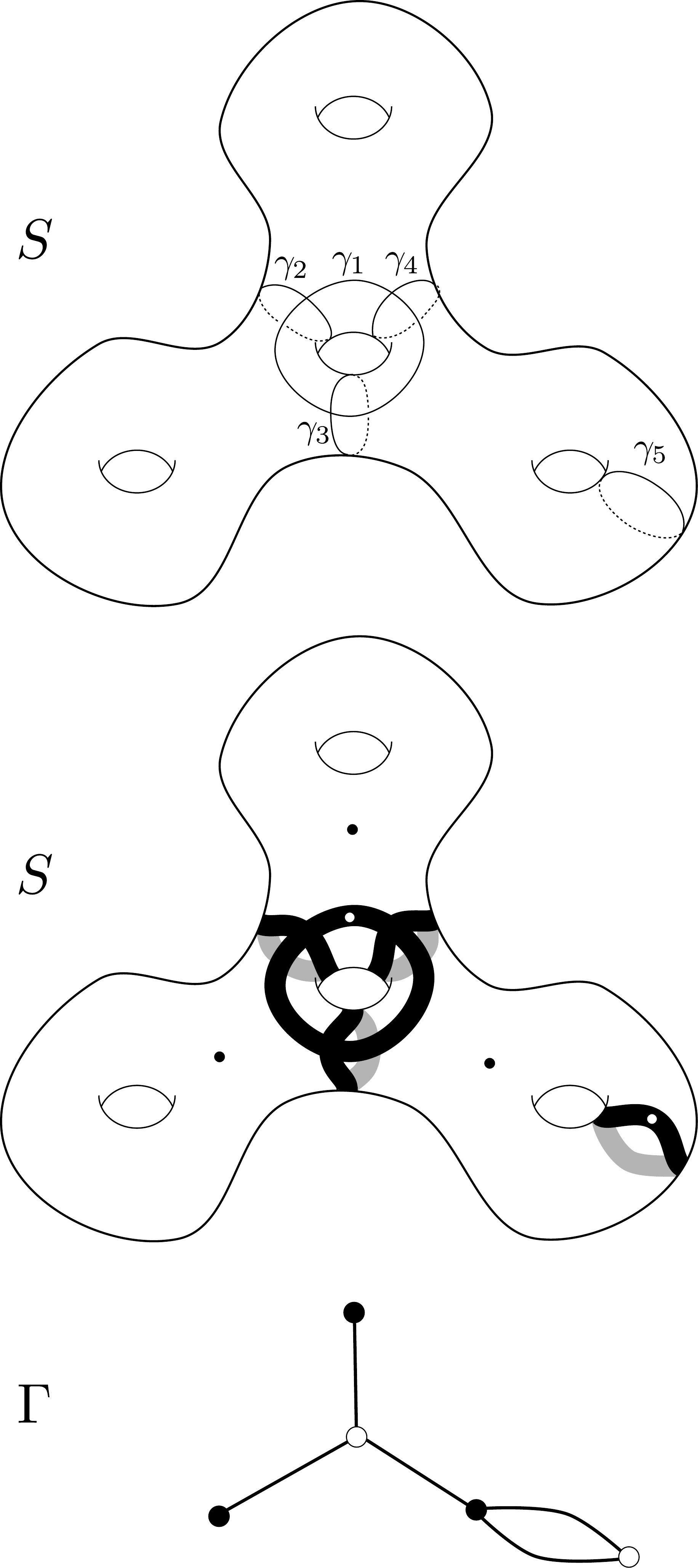}
	\end{minipage}
	\hfill
	\begin{minipage}[h]{0.48\linewidth}
		\caption[Regular neighborhood]{
		\label{fig:Gtwo-123}
		Let $\sigma_1$, $\sigma_2$, $\sigma_3$, $\sigma_4$, and $\sigma_5$ denote the induced splittings of the fundamental group of~$S$.
		Shown in bold is a PL regular neighborhood of $\{\gamma_1, \gamma_2, \gamma_3\, \gamma_4, \gamma_5\}$;
		call it $N$.
		Construct the dual graph to the frontier of $N$,
		making a $V_0$-vertex for each component of $N$ and a $V_1$-vertex for each component of $S - N$.
		The resulting graph, call it $\Gamma$, is bipartite.
		$\Gamma$~has two $V_0$-vertices.
		One $V_0$-vertex encloses $\sigma_1$, $\sigma_2$, $\sigma_3$, and $\sigma_4$.
		The other $V_0$-vertex encloses $\sigma_5$.
		The preimage of $\Gamma$ in $\widetilde{S}$
		is an algebraic regular neighborhood of $\{\sigma_1, \sigma_2, \sigma_3, \sigma_4, \sigma_5\}$.	
		}
	\end{minipage}
\end{figure}

\section{Symmetry of Crossing}
\label{Symmetry}

Let $G$ be any group with subgroups $H$ and $K$.
Let $X$ and $Y$ be $H$- and $K$-almost invariant subsets of $G$, respectively. Recall from Page \pageref{DEFNcrossing} the definition of crossing:
\begin{center}
	$X$ crosses $Y \iff $ all four corners of the pair $(X,Y)$ are $K$-infinite.
\end{center}
If $G$ is finitely generated and neither $X$ nor $Y$ is trivial,
an argument using coboundary in the Cayley graph for $G$ shows that the relation
``$X$ crosses $Y$'' is symmetric (see Lemma 2.3 of \cite{Scott1998}).
For non-finitely generated $G$, this argument utterly fails,
and so it seems plausible that crossing of almost invariant sets is not symmetric.
However, below I prove that if $X$ and $Y$come from splittings of $G$,
then crossing is symmetric.
Here is the key lemma.

\begin{lem}
\label{HANDKFINITE}
Let $G$ be any group with subgroups $H$ and $K$.
Suppose $Y$ is a \mbox{$K$-almost} invariant set arising from a splitting of $G$ over $K$.
Further, suppose $Y$ contains some nonempty subset $X'$ that is stabilized by $H$
(equivalently, $Y$ contains at least one right $H$-coset).
Take any $g_0 \in Y^*$. Then
\[
	H g_0 \cap Y^*
	= (H \cap K)g_0 \cap Y^*
\]
(so that $H g_0 \cap Y^*$ is both $H$- and $K$-finite).
\end{lem}

\begin{proof}
Clearly $(H \cap K) g_0 \cap Y^*$ is a subset of $H g_0 \cap Y^*$.
In the remaining part of the proof, we show that $H g_0 \cap Y^*$ is a subset of $(H \cap K) g_0 \cap Y^*$.

Since $Y$ is a standard almost invariant set arising from a splitting, there exists a $G$-tree $T$
with an edge $e$ and a vertex $w$,
such that $K = Stab(e)$ and
\[
	Y = \{g \in G | e \text{ points away from } gw \}.
\]
We measure the distance between two vertices in $T$ by counting the number of edges in a simple path connecting them.
Since $H g_0 \cap Y^*$ is contained in a single $H$-coset, and since $X$ is stabilized by $H$,
we can choose $x \in X'$ such that the path $[xw, g_0w]$ realizes the minimum distance from $X'w$ to $(H g_0 \cap Y^*)w$.
Let $D$ denote this distance.
See Figure~\ref{fig:TreeHg0-1} for the basic picture.

Take any $h \in H$ such that $h g_0 \in Y^*$.
We will show that $h \in K$.
Since $h$ stabilizes $X'$,
multiplying the path $[xw, g_0w]$ on the left by $h$ gives another path from $X'w$ to $(H g_0 \cap Y^*)w$ of length $D$.
As $X' \subset Y$, both paths must pass pass through $e$.
Let $d_1$ denote the distance from $xw$ to $e(0)$.
Let $d_2$ denote the distance from $hxw$ to $e(0)$.
I claim that $d_1 = d_2$.
If $d_1 < d_2$, then $[xw, hg_0w]$ would be a path from $X'w$ to $(H g_0 \cap Y^*)w$ with length strictly less than $D$.
Similarly, if $d_2 < d_1$, then $[hxw, g_0w]$ would be a path from $X'w$ to $(H g_0 \cap Y^*)w$ with length strictly less than $D$.
Hence $d_1 = d_2$.
It follows that $e = he$.
As $K$ is the stabilizer of $e$, this implies $h \in K$, as desired.
This concludes the proof that $H g_0 \cap Y^*$ is equal to $(H \cap K) g_0 \cap Y^*$.
\end{proof}

\begin{figure}[!htbp]
	\centering
	\begin{minipage}[h]{0.48\linewidth}
		\includegraphics[width=65mm]{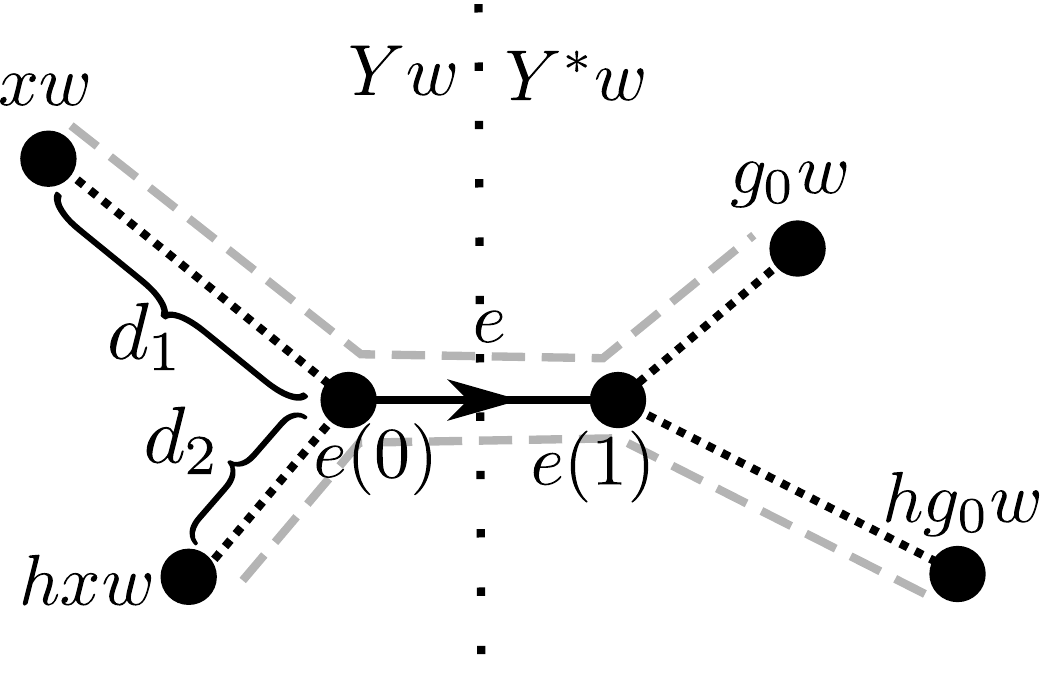}
	\end{minipage}
	\hfill
	\begin{minipage}[h]{0.48\linewidth}
		\caption[Proof of Key Lemma]{
		\label{fig:TreeHg0-1}
		Proof of Lemma~\ref{HANDKFINITE}.
		}
	\end{minipage}
\end{figure}

We can use the lemma to prove symmetry of crossings for almost invariant sets that come from splittings.

\begin{prop}
\label{SYMMETRIC}
Let $G$ be any group with subgroups $H$ and $K$.
Suppose $X$ is any nontrivial $H$-almost invariant set,
and $Y$ is a standard $K$-almost invariant set arising from a splitting of $G$ over $K$.
If $X$ crosses $Y$, then $Y$ crosses $X$.
\end{prop}

\begin{proof}
If $Y$ does not cross $X$, then one of the corners of $(X,Y)$ is $H$-finite.
Without loss of generality (after possibly replacing $X$ by $X^*$ or $Y$ by $Y^*$),
$X \cap Y^*$ is $H$-finite.
This means that we can choose finitely many $g_i \in X \cap Y^*$ such that
$X \subset Y \cup H g_1 \cup \ldots \cup H g_r$.
Let $X' := X - \coprod\limits_{i=1}^{r} H g_i$, so that $X' \subset Y$.
Since $X$ is nontrivial, $X'$ is nonempty.
$X'$ is also stabilized by $H$.
Whereas $Y$ comes from a splitting,
\ref{HANDKFINITE} proves that $Hg_i \cap Y^* = (H \cap K)g_i \cap Y^*$, for all $i$.
Hence $X \cap Y^*$ is $K$-finite, so that $X$ does not cross $Y$.
\end{proof}

\begin{cor}
Intersection number of a pair of splittings (as defined on Page \pageref{DEFNintersectionnumber}) is well-defined,
even if the ambient group is not finitely generated.
\end{cor}

\section{Examples of Infinite Intersection Number}
\label{IntersectionNumber}

Scott and Swarup have shown that the intersection number of the two splittings
of a finite group over finitely generated subgroups is finite
(see Lemma 2.7 of \cite{Scott1998}).
In the spirit of this paper, one might ask if we can eliminate the hypothesis of $G$ being finitely generated.
The answer is, definitively, ``no.''
In fact there exist free product splittings induced by simple curves on surfaces,
with infinite intersection number:

\begin{exmp}
\label{INFINTERNUMBER}
Let $S$ be an infinite strip with countably many punctures:
\[
	S := [-\frac{1}{2}, \frac{1}{2}] \times \mathbb{R} - \{0\} \times \mathbb{Z}.
\]
Take $l_1$ and $l_2$ as shown in Figure~\ref{fig:Int-inf}.
Let $l_1^+$ be a regular neighborhood of the part of $S$ lying above $l_1$,
and define $l_1^-$, $l_2^+$, and $l_2^-$ similarly.
By Van Kampen's theorem, we have the following two splittings of $G := \pi_1(S)$ over the trivial group.
\[
	\sigma: G \cong \pi_1(l_1^+) *_{\{1\}} \pi_1(l_1^-)
	= \mathbb{F}^\mathbb{Z} * \mathbb{F}^\mathbb{Z}
\]
\[
	\tau: G \cong \pi_1(l_2^+) *_{\{1\}} \pi_1(l_2^-)
	= \mathbb{F}^\mathbb{Z} * \mathbb{F}^\mathbb{Z}
\]
Here, the intersection number of $\sigma$ and $\tau$ is visibly infinite.
\end{exmp}

\begin{figure}[!htbp]
	\centering
	\begin{minipage}[h]{0.48\linewidth}
		\includegraphics[width=75mm]{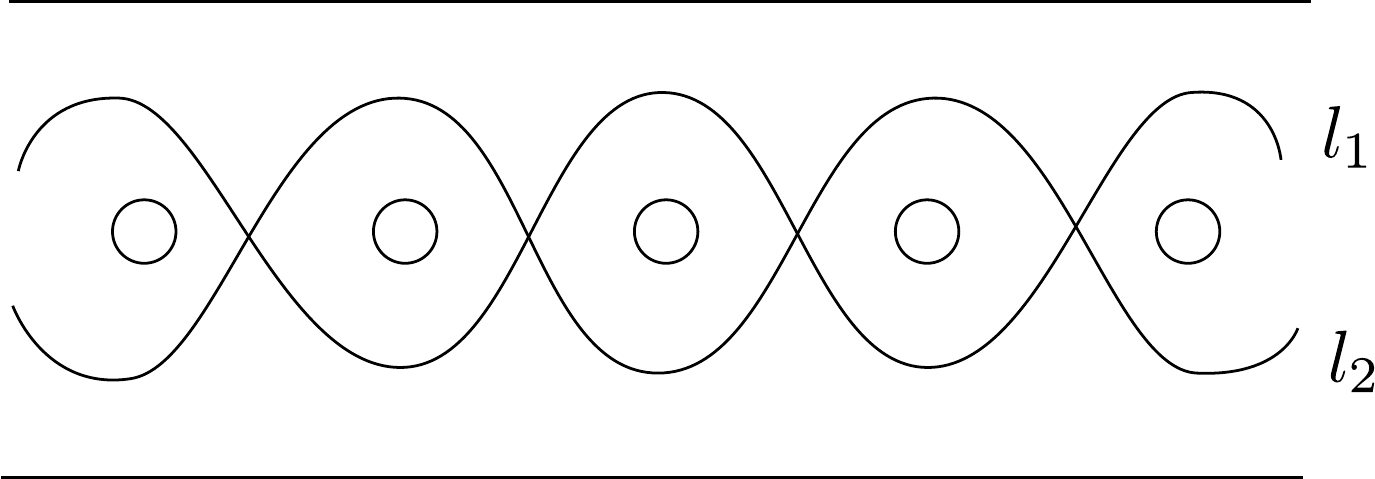}
	\end{minipage}
	\hfill
	\begin{minipage}[h]{0.48\linewidth}
		\caption[Infinite intersection number]{
		\label{fig:Int-inf}
		Two curves yielding a pair of splittings with infinite intersection number.
		}
	\end{minipage}
\end{figure}

We have a similar example exhibiting infinite self-intersection number
for an almost invariant set not that is not associated to a splitting.

\begin{exmp}
\label{INFSELFINTER}
Take $S$ and $G$ as in the previous example.
We can find a $\{1\}$-almost invariant subset of $G$ with infinite self-intersection number.
See Figure~\ref{fig:Int-selfinf}.
\end{exmp}

\begin{figure}[!htbp]
	\centering
	\begin{minipage}[h]{0.48\linewidth}
		\includegraphics[width=75mm]{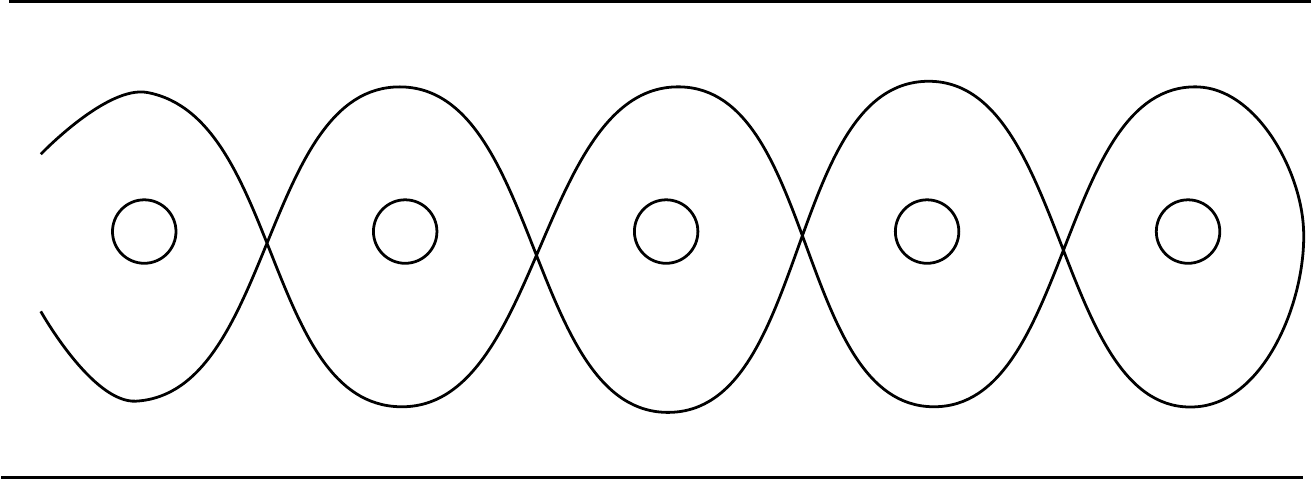}
	\end{minipage}
	\hfill
	\begin{minipage}[h]{0.48\linewidth}
		\caption[Infinite self-intersection number]{
		\label{fig:Int-selfinf}
		A curve yielding a $\{1\}$-almost invariant set with infinite self-intersection number.
		}
	\end{minipage}
\end{figure}

\section{Almost Inclusion}
\label{GoodPosition}

Let $\{X_j | j \in J\}$ be a collection of $H_j$-almost invariant subsets of a group $G$,
and let $\Sigma := \{gX_j, gX_j^* | g \in G, j \in J \}$.
In this section, we prove that $\leq$ (from Definition~\ref{DEFNPO}) defines a partial order on $\Sigma$.

The following lemma, proved in a preprint by Scott and Swarup,
shows that if an $H$-almost invariant set is $H$-almost equal to a $K$-almost invariant set,
then $H$ and $K$ are commensurable.

\begin{lem}
\label{COMMENSURABLE}
Let $G$ be any group with a nontrivial $H$-almost invariant set $X$ and a nontrivial $K$-almost invariant set $Y$.
If $X \overset{H-a}{=} Y$, then $H$ and $K$ are commensurable subgroups of $G$.
\end{lem}

\begin{proof}
$X \overset{H-a}{=} Y$ immediately implies that
$Xg  \overset{H-a}{=} Yg$, for all $g \in G$.
As $X$ is $H$-almost invariant, we have
$Xg \overset{H-a}{=} X$, and hence
$Yg \overset{H-a}{=}  Y,$ for all $g \in G$.
As $Y$ is $K$-almost invariant,
$K$ stabilizes $Y$, so each of $Y$ and $Y^*$ is a union of cosets $Kg$ of $K$ in $G$.

Since $Y$ is nontrivial, we can choose $u, v \in G$ such that
$Ku \subset Y$ (equivalently, $K \subset Y u^{-1}$) and
$Kv \subset Y^*$.
Recall that by the preceding paragraph, $Y(u^{-1}v) \overset{H-a}{=} Y$.
Note that $Kv$ lies in the symmetric difference of $Y(u^{-1}v)$ and $Y$,
so $Kv$ must be $H$-finite.
Hence $K$ is also $H$-finite.
We can write $K \subset \coprod\limits_{i=1}^r H g_i$, where $r$ is minimal.
We have $K = \coprod\limits_{i=1}^r (K \cap H) g_i$.
As $K$ is the union of finitely many $(K \cap H)$-cosets,
it follows that $[K: K \cap H] < \infty$.

A similar argument shows that $K \cap H$ is finite index in $H$. 
Hence $H$ and $K$ are commensurable subgroups of $G$.
\end{proof}

Next we show that if $X$ and $Y$ arise from splittings,
then their stabilizers are actually equal.

\begin{lem}[Modified from Lemma 2.2 of \cite{ScottSwarup2000}]
\label{EQUIV}
Let $X$ be an $H$-almost invariant subset arising from a splitting of $G$ over $H$,
and let $Y$ be a $K$-almost invariant set arising from a splitting of $G$ over $K$.
\begin{enumerate}
	\item
	If two corners of the pair $(X,Y)$ are $H$-finite, then $H = K$.
	\item
	If two corners of the pair $(X,gX)$ are $H$-finite, then $g$ normalizes $H$.
\end{enumerate}
\end{lem}

\begin{proof}
To prove the first part of the lemma, suppose two corners of the pair $(X,Y)$ are $H$-finite.
Without loss of generality, $X \cap Y^*$ and $X^* \cap Y$ are $H$-finite
(if not, replace $X$ by $X^*$),
so that $X$ and $Y$ are $H$-almost equal.

As $Y$ is a standard almost invariant set arising from a splitting,
$hY$ and $Y$ are nested, for all $h \in G$.
We will now show that $H \subset K$.
Let $h \in H$. If $hY \subset Y^*$ (or $hY^* \subset Y$),
then $X \overset{H-a}{=} X^*$, a contradiction to $G$ being $H$-infinite.
If $hY \subset Y$ but $hY \not= Y$, then we get an infinite chain of inclusions
\[
	\ldots \subset h^nY \ldots \subset hY \subset Y.
\]
As $H$ and $K$ are commensurable, some power of $h$ lies in $K$, so that $h^n Y = Y$ for some $n$.
This implies $hY = Y$, so that $h \in K$.
Similarly, if $Y \subset hY$, we must also have $Y = hY$ and $h \in K$.

A similar argument shows that $K \subset H$. Hence $H = K$.

To prove the second part of the lemma, apply the first part using $Y := gX$.
The first part of the lemma gives $H = K = gHg^{-1}$, so that $g$ normalizes $H$.
\end{proof}

Now we show that if $X$ and $Y$ are standard $H$- and $K$-almost invariant sets arising from non-isomorphic splittings of $G$,
then it is impossible to have $X \overset{H-a}{=} Y$.

\begin{prop}[Modified from Lemma 2.3 of \cite{ScottSwarup2000}]
\label{EQUIVALENT}
Let $X$ and $Y$ be $H$- and $K$-almost invariant sets
arising from splittings $\sigma$ and $\tau$ of $G$ over subgroups $H$ and $K$, respectively.
If two corners of the pair $(X,Y)$ are small, then $\sigma$ and $\tau$ are isomorphic splittings.

Further, at least one of the following holds (after possibly replacing $X$ by $X^*$):
\begin{enumerate}
	\item
	$X \mapsto Y$ induces a $G$-equivariant, order-preserving isomorphism
	\\
	from $(\Sigma(X), \subset)$ to $(\Sigma(Y), \subset)$; or
	\item
	The two splittings are of the form $G = A *_H B$, where $H$ has index $2$ in $A$,
	and there exists $a \in A$ such that $X \mapsto aY$ induces a
	$G$-equivariant, order-preserving isomorphism from $\Sigma(X)$ to $\Sigma(Y)$.
\end{enumerate}
\end{prop}

\begin{proof}
By replacing $X$ by $X^*$ if necessary, without loss of generality,
$X \cap Y^*$ and $X^* \cap Y$ are $H$-finite,
i.e. $X \overset{H-a}{=} Y$.
By Lemma~\ref{EQUIV}, we have $H = K$.

A corner of $(X, gX)$ is small if, and only if, the corresponding corner of $(Y, gY)$ is small.
If, for all $g \in G - H$, only one corner of $(X, gX)$ is small (and hence empty),
then the corresponding corner of $(Y, gY)$ must also be empty.
Then $X \mapsto Y$ induces a $G$-equivariant, order-preserving isomorphism
from $\Sigma(X)$ to $\Sigma(Y)$, and hence the splittings are isomorphic by Dunwoody's theorem
(see Section~\ref{SUBdunwoody}).

If there exists $g \in G - H$ such that two corners of $(X, gX)$ are small,
then the trees for $\sigma$ and $\tau$ must each have some vertices of valence two.
There are two cases:
\begin{enumerate}
	\item 
	$\sigma$ is a trivially ascending HNN extension, $G \cong H *_H$.
	Then $T_\sigma$ and $T_\tau$ are lines, $H$ and $K$ are normal in $G$,
	and $G = \langle H, t \rangle$ for some $t \in G$.
	Thus $X \mapsto Y$ induces a $G$-equivariant, order-preserving isomorphism
	from $\Sigma(X)$ to $\Sigma(Y)$.

	\item
	$\sigma$ is an amalgamated free product of the form $G \cong A *_H B$,
	where $H$ has index $2$ in $A$.
	We can write $A = \langle H, a \rangle$.
	Then $X \overset{H-a}{=} aX^*$ and $Y \overset{H-a}{=} aY^*$.
	If the corresponding corners of $(X, aX^*)$ and $(Y, aY^*)$ are empty,
	then $X \mapsto Y$ induces a $G$-equivariant, order-preserving isomorphism
	from $\Sigma(X)$ to $\Sigma(Y)$, so we're done.
	Otherwise, $X \mapsto aY$ does the trick.
\end{enumerate}

\end{proof}

\begin{cor}
\label{GOODPOSITION}
Let $G$ be any group with any collection $\{\sigma_j | j \in J\}$ of pairwise non-isomorphic splittings.
For each $j$, let $X_j$ be a standard $H_j$-almost invariant set arising from $\sigma_j$.
Then $\Sigma := \mycup \limits_{j \in J} \Sigma(X_j)$ is in good position.
\end{cor}

\begin{proof}
If there exists $g \in G$ and distinct $j,k \in J$ such that two corners of the pair $(X_j, g X_k)$ are small,
then $\sigma_j$ and $\sigma_k$ are isomorphic splittings (by Proposition~\ref{EQUIVALENT}),
a contradiction to the hypotheses.
\end{proof}

Since $\Sigma$ is in good position, we can define a partial order on $\Sigma$ as follows.

\begin{cor}
\label{PARTIALORDER}
Let $G$ be any group with any collection $\{\sigma_j | j \in J\}$ of pairwise non-isomorphic splittings.
For each $j$, let $X_j$ be an $H_j$-almost invariant set arising from $\sigma_j$.
Let $\Sigma := \mycup \limits_{\sigma \in J} \Sigma(X_j)$.
Define a binary relation $\leq$ on $\Sigma$ by
\[
	A \leq B \Leftrightarrow A \cap B^* \text{ is empty or the only small corner of the pair } (A, B).
\]
Here ``small'' means ``$Stab(A)$-finite'' or equivalently ``$Stab(B)$-finite'' (see Proposition~\ref{SYMMETRIC}).
Then $\leq$ defines a partial order on $\Sigma$.
\end{cor}

\begin{proof}
Reflexivity is obvious. We need to show antisymmetry and transitivity.

To show antisymmetry, suppose $A \leq B$ and $B \leq A$.
Then both $A \cap B^*$ and $B \cap A^*$ are small corners of the pair $(A, B)$.
Since two corners are small, the first inequality now implies $A \cap B^*$ is empty,
while the second implies $B \cap A^*$ is empty.
Hence $A = B$.
Thus $\leq$ satisfies antisymmetry.

To show transitivity, suppose $A \leq B$ and $B \leq C$, where $A$, $B$, and $C$ are all distinct.
We need to show that $A \leq C$.
Since $B \leq C$, we can subtract finitely many $Stab(B)$-cosets from $B$ to obtain $B' \subset C$.
Since $A \leq B$ and $B$ is $Stab(B)$-almost equal to $B'$,
we have $A \cap {B'}^*$ is $Stab(B)$-finite.
By Lemma~\ref{HANDKFINITE}, since $A$ arises from a splitting, $A \cap {B'}^*$ is also $Stab(A)$-finite.
Hence we can subtract finitely many $Stab(A)$-cosets from $A$ to obtain $A' \subset B'$.
It follows that $A' \subset C$.
Since $A$ is $Stab(A)$-almost equal to $A'$ and since $A' \subset C$,
we have $A \cap C^*$ is a small corner of the pair $(A,C)$.

Thus the only way we could possibly fail to have $A \leq C$ is if another corner were small.
If two corners of the pair $(A,C)$ are small, then
\ref{EQUIVALENT} proves that $A$, $B$ and $C$ all must have come from isomorphic splittings of $G$.
Since we assumed no two distinct $j$'s have isomorphic $\sigma_j$'s,
it follows that $A$, $B$ and $C$ are all translates of $X_j$ or $X_j^*$, for the same $j$.
So we must have $A \subset B \subset C$. This completes the proof that $\leq$ satisfies transitivity.
\end{proof}

Now that we've put a partial order $\leq$ on $\Sigma$, we show that the partial order is unique.

\begin{cor}[Uniqueness of the partial order]
\label{UNIQUENESSPO}
Let $G$ be any group with any collection $\{\sigma_j | j \in J\}$ of pairwise non-isomorphic splittings.
Suppose that $\{\sigma'_j | j \in J\}$ is another collection of splittings of $G$, where $\sigma_j \cong \sigma'_j$, for all $j \in J$.
For each $j$, let $X_j$ be a standard $H_j$-almost invariant set arising from $\sigma_j$,
and let $X'_j$ be a standard $H'_j$-almost invariant set arising from $\sigma'_j$.
Let $\Sigma := \mycup \limits_{j \in J} \Sigma(X_j)$,
and let $\Sigma' := \mycup \limits_{j \in J} \Sigma(X'_j)$.
Then there exists a $G$-equivariant, order-preserving isomorphism
from $(\Sigma, \leq)$ to $(\Sigma ', \leq)$.
\end{cor}

\begin{proof}
By Proposition~\ref{EQUIVALENT}, for all $j$, there exist $g_j \in G$ such that
$X_j \mapsto g_j X'_j$ or $X^*_j \mapsto g_j X'_j$
induces a $G$-equivariant, order-preserving isomorphism
from $\Sigma(X_j)$ to $\Sigma(X'_j)$.
Together, these induce a $G$-equivariant isomorphism from $\Sigma$ to $\Sigma '$.
We need to show that this isomorphism is order-preserving.
As no two of the $\sigma_j$'s are isomorphic,
whenever $A \in \Sigma (X_j)$ and $B \in \Sigma (X_k)$ ($k \neq j$),
at most one corner of $(A, B)$ is small.
If no corner of $(A,B)$ is small, then no corner of $(A', B')$ is small.
If exactly one corner of $(A, B)$ is small,
then the same corner of $(A',B')$ must be the only small corner of $(A', B')$.
Hence $\Sigma \rightarrow \Sigma'$ is order-preserving.
\end{proof}

Next we spell out this uniqueness result in the case when the splittings happen to be compatible.
In this case, we allow some of the $\sigma_j$'s to be isomorphic to each other.

\begin{cor}[Uniqueness of compatibility trees]
\label{UNIQUENESSTREE}
Let $G$ be any group with a finite collection $\{\sigma_j | j \in J\}$ of splittings.
Suppose $\{\sigma_j | j \in J\}$ is compatible, and
let $T$ and $T'$ be compatibility trees.
Then there exists a $G$-equivariant isomorphism from $T$ to $T'$.
\end{cor}

\begin{proof}
First, we prove the result in the case where no two distinct $j$'s have isomorphic splittings.
Fix a vertex $v$ in $T$. For each $j$, pick a $\sigma_j$-edge $e_j$ in $T$,
and define a subset $X_j$ of $G$ by
\[
	X_j := \{g \in G | e \text{ points away from } gv\}.
\]
Fix a vertex $v'$ in $T'$. For each $j$, pick a $\sigma_j$-edge $e'_j$ in $T'$ whose stabilizer is the same as $Stab(e_j)$,
such that
\[
	X_j \overset{H_j-a}{=} \{g \in G | e'_j \text{ points away from } gv'\}
\]
(we can do this by Lemma~\ref{BASICEQUIV}).
Let $X'_j$ denote the set $\{g \in G | e'_j \text{ points away from } gv'\}$.
Apply Corollary~\ref{UNIQUENESSPO} to get a $G$-equivariant, order preserving isomorphism
from $\mycup \limits_{j \in J} \Sigma(X_j)$ to $\mycup \limits_{j \in J} \Sigma(X'_j)$.
Dunwoody's theorem (see Section~\ref{SUBdunwoody})
now gives a $G$-equivariant isomorphism from $T$ to $T'$.

Second, we prove the result in the case where $\{\sigma_j | j \in J\}$ possibly has duplicate splittings.
For each isomorphism class $\{\sigma_j | j \in I\}$ of splittings, discard all but one representative; call it $\sigma_I$.
Note that the edge in $T$ (or $T'$) corresponding to $X_I$
must be contained in an interval of $|I|$ edges, one for each $\sigma_j$ in the isomorphism class,
where the interior vertices of the interval each have valence two.
Collapse the edge orbits of $T$ and $T'$ corresponding to the discarded splittings.
To recover an isomorphism fro $T$ to $T'$, for each isomorphism class $I$,
subdivide each $\sigma_I$ edge in the collapsed $T$ and the collapsed $T'$ into an interval of $|I|$ edges.
\end{proof}

\section{Compatibility and Intersection Number Zero}
\label{IntersectionZero}

Take any finite collection of non-isomorphic splittings of $G$ satisfying sandwiching (see Definition~\ref{DEFNSANDWICHING}).
Here we show that if the splittings have pairwise intersection number zero,
then the splittings are compatible (this is Theorem~\ref{COMPATIBLE}).
This is a special case of very good position (see Section~\ref{VeryGoodPosition}),
when the intersection number of each pair of splittings is zero.

The sandwiching assumption is necessary;
see Section~\ref{NECESSARY}.
For more intuition about sandwiching, we begin by proving that proving that sandwiching is automatic
if none of the splittings is a trivially ascending HNN extension (see Page \pageref{TRIVASC}).
The key fact used is that if $X$ is a standard almost invariant set arising from a splitting that is not trivially ascending HNN,
then all four types of nesting occur between $X$ and its translates:

\begin{lem}
Let $\sigma$ be a splitting of $G$ over $H$, where $\sigma$ is not a trivially ascending HNN extension.
Let $X$ be a standard $H$-almost invariant set arising from $\sigma$.
Then, by varying $g$, all four of $gX^{(*)} \subset X^{(*)}$ occur.
\label{ALLINCLUSIONS}
\end{lem}

This result is a strengthening of Lemma 5.5 of \cite{ScottSwarup2003}, which assumes that $\sigma$ is not any ascending HNN extension.

\begin{proof}
Since $X$ is a standard almost invariant set arising from $\sigma$, there is a $G$-tree $T$ with an edge $e$ and a vertex $w$,
and exactly one orbit of edges, such that
\[
	X = \{ g \in G | e \text{ points away from } gw\}.
\]
 It suffices to show that there exist translates of $e$ such that
 $g_1 e < g_2 e$ and $g_3 e < g_4 \overline{e}$.
 There are two cases:

\begin{enumerate}
	\item
	$T$ is a line, so since $\sigma$ is not trivially ascending HNN,
	$\sigma$ must have the form $G \cong A*_H B$, where $|A:H| = |B:H| = 2$.
	To get $g_1 e < g_2 e$, take two translates of $e$ separated by 1 edge.
	To get $g_3 e < g_4 \overline{e}$, take two adjacent translates of $e$.
	\item
	$T$ has branching, so we can find three distinct translates of $e$
	such that the geodesics between any two pair of them all meet at exactly one vertex,
	and that either two of the translates point toward the vertex and one points away, or vice-versa.
	See Figure~\ref{fig:branching}.
	To get $g_1 e \leq g_2 e$, take two of these translates of $e$
	where one is pointing toward the vertex and the other away.
	To get $g_3 e \leq g_4 \overline{e}$,
	take two translates pointing toward (or two pointing away from) the vertex.
\end{enumerate}
\end{proof}

\begin{figure}[!htbp]
	\centering
	\begin{minipage}[h]{0.48\linewidth}
		\includegraphics[width=65mm]{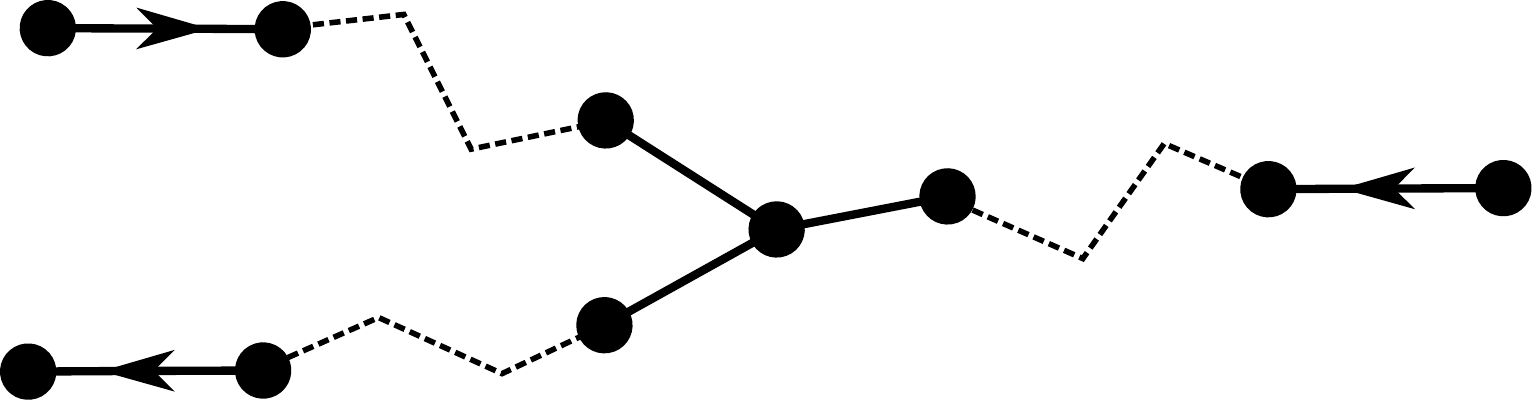}
	\end{minipage}
	\hfill
	\begin{minipage}[h]{0.48\linewidth}
		\caption[Branching]{
		\label{fig:branching}
		If a $G$-tree $T$ has branching and exactly one edge orbit,
		then for any edge, we can find three of its translates such that
		either two point toward each other and the other one points away, or vice versa.
		}
	\end{minipage}
\end{figure}

If $X$ and $Y$ are standard almost invariant sets arising from splittings of $G$,
then either $X$ crosses all translates of $Y$,
or all four types of almost nesting occur between $X$ and translates of $Y$:

\begin{lem}
\label{ALLCOMPARISONS}
Let $\sigma$ and $\tau$ be splittings of $G$ over $H$ and $K$, respectively,
where neither $\sigma$ nor $\tau$ is a trivially ascending HNN extension.
Let $X$ and $Y$ be standard almost invariant sets arising from $\sigma$ and $\tau$, respectively.
Suppose that there exists $g_0 \in G$ such that $X$ and $g_0 Y$ do not cross.
Then, by varying $g \in G$, all four of $X^{(*)} \leq g Y^{(*)}$ occur.
\end{lem}

\begin{proof}
Without loss of generality (after possibly replacing $X$ by $X^*$ or $Y$ by $Y^*$),
$X \leq g_0Y$.
Obtain each of the four cases as follows:
\begin{enumerate}
	\item
	$X \leq g_0Y$ is already given.
	\item
	To show there exists $g$ with $X \leq g Y^*$,
	apply Lemma~\ref{ALLINCLUSIONS} to get $g_0 Y \subset g_1 Y^*$,
	\\
	so that $X \leq g_0 Y \subset g_1 Y^*$.
	\item
	To show there exists $g$ with $X^* \leq g Y^*$,
	apply Lemma~\ref{ALLINCLUSIONS} to get $g_2 X^* \subset X$.
	\\
	Now $g_2 X^* \subset X \leq g_0Y \subset g_1 Y^*$,
	so that $X^* \leq {g_2}^{-1} g_1 Y^*$
	\item
	To show there exists $g$ with $X^* \leq gY$,
	apply Lemma~\ref{ALLINCLUSIONS} to get $g_0Y \subset g_3Y$.
	\\
	Now $g_2 X^* \subset X \leq g_0Y \subset g_3Y$,
	so that $X^* \leq {g_2}^{-1} g_3 Y$.
\end{enumerate}
\end{proof}

\begin{rem}
If we assume that $X$ and $g_0 Y$ are nested (instead of almost nested),
then the same proof shows that all four inclusions $X^{(*)} \subset g Y^{(*)}$ occur.
\end{rem}

\begin{cor}
\label{TRIVHNN}
Let $G$ be any group with any collection $\{\sigma_j | j \in J\}$ of splittings,
where no $\sigma_j$ is a trivially ascending HNN extension.
Then $\{\sigma_j | j \in J\}$ satisfies sandwiching.
\end{cor}

Most of the results in the rest of the paper will require the sandwiching assumption.
The key reason we need sandwiching is to get interval finiteness:

\begin{prop}
\label{INTERVALS}
Let $\sigma_j$ be a splitting of $G$ over $H_j$,
and assume $\{\sigma_1, \ldots, \sigma_n\}$ satisfies sandwiching.
Let $X_j$ be an $H_j$-almost invariant set arising from $\sigma_j$.
Let $\Sigma = \{ gX_j, gX_j^* | g \in G, j = 1, \ldots, n\}$.
Then for all $A, B \in \Sigma$,
there are only finitely many $C \in \Sigma$ such that $A \leq C \leq B$.
\end{prop}

\begin{proof}
Fix $A, B \in \Sigma$.
If $A \nleq B$, then there is no $C$ such that $A \leq C \leq B$;
so assume $A \leq B$.
Since $\mycup \limits_{j \in J} \Sigma(X_j)$ satisfies sandwiching,
for each $j \in J$ we can choose $A_j, B_j \in \Sigma(X_j)$ such that 
\[
	A_j \leq A \leq B \leq B_j.
\]
If $C \in \Sigma(X_j)$ and $A \leq C \leq B$, then $A_j \leq C \leq B_j$.
But for each $j$, there are only finitely many such $C$.
As we are only considering finitely many splittings,
there are only finitely many $C \in \Sigma$ satisfying $A \leq C \leq B$.
\end{proof}

\begin{thm}
\label{COMPATIBLE}
Let $\sigma_j$ be a splitting of $G$ over $H_j$
and assume $\{\sigma_1, \ldots, \sigma_n\}$ satisfies sandwiching.
Let $X_j$ be an $H_j$-almost invariant set arising from $\sigma_j$.
If $i(\sigma_j, \sigma_k) = 0$, for all $j$ and $k$,
then the splittings $\{ \sigma_1, \ldots, \sigma_n\}$ are compatible.
\end{thm}

\begin{proof}
First we prove the theorem for the case when no $\sigma_j$ is isomorphic to any other.

By Corollary~\ref{GOODPOSITION}, $\leq$ is a partial order on
$\Sigma = \{ gX_j, gX_j^* | g \in G, j = 1, \ldots, n \}$.
We can see that the four conditions of Dunwoody's theorem (see Section \ref{SUBdunwoody}) are satisfied:
\begin{enumerate}
	\item
	For all $A, B \in \Sigma$, if $A \leq B$, then $B^* \leq A^*$.
	\item
	For all $A, B \in \Sigma$ with $A \leq B$,
	there are only finitely many $C \in \Sigma$ with $A \leq C \leq B$ (see Proposition\ref{INTERVALS}).
	\item
	For all $A, B \in \Sigma$, at least one of $A^{(*)} \leq B^{(*)}$
	(because $i(\sigma_j, \sigma_k) = 0$, for all $j$ and $k$).
	\item
	We cannot have simultaneously $A \leq B$ and $A \leq B^*$.
\end{enumerate}
Construct Dunwoody's tree,
$T_{\Sigma}$, with edge set $\Sigma$.
Each edge is a $\sigma_j$-edge for unique $j$.
For all $j$, we have a $G$-equivariant isomorphism
$T_\Sigma / (\text{all but } j \text{-edges collapsed}) \rightarrow T_{\sigma_j}$.
Hence $T_\Sigma$ is a compatibility tree for  $\{ \sigma_1, \ldots, \sigma_n\}$.

Second, we prove the theorem in the case when we possibly have duplicate splittings.
Discard all but one splitting from each isomorphism class.
Apply the above procedure.
Then subdivide the resulting tree, as in the proof of Corollary~\ref{UNIQUENESSTREE}.
\end{proof}

\section{$CAT(0)$ Cubical Complexes and Positive Intersection Number}
\label{VeryGoodPosition}

A {\bf cubical complex} $C$ is a CW-complex whose cells are standard Euclidean cubes of varying dimensions,
such that the intersection of any two cells is either empty or a common face of both.
$C$ is called a {\bf $CAT(0)$ cubical complex} if, in addition,
$C$ is simply connected, and the link of any vertex (i.e. $0$-cube) is a flag complex.
Another word for ``$CAT(0)$ cubical complex'' is ``cubing.''

In this section, we start with any finite collection of pairwise non-isomorphic splittings of any group $G$,
and construct a $CAT(0)$ cubical complex.
$G$ acts naturally on the complex, and each hyperplane orbit will correspond to one of the splittings.
Furthermore, hyperplanes cross precisely when their associated splittings cross.
Essentially, we are showing how to make Niblo-Sageev-Scott-Swarup's ``minimal cubing'' construction
from~\cite{NibloSageevScottSwarup2005} work
without requiring $G$ or the subgroups over which $G$ splits to be finitely generated.
Their ``minimal cubing'' construction, in turn, was a generalization
of Sageev's cubing construction in~\cite{Sageev1995}.
For applications of the cubing construction, see Theorem~\ref{VGP} and Section~\ref{AlgebraicRegularNeighborhoodsE}.

We will briefly review all the basic constructions.
For more details, see~\cite{NibloSageevScottSwarup2005} and Sections $2$ and $3$ of~\cite{Sageev1995}.

\subsection{Producing almost invariant sets from a $CAT(0)$ cubical complex}

In~\cite{ScottSwarup2003}, Scott and Swarup showed how to produce a an almost invariant set from a $G$-tree.
Niblo, Sageev, Scott and Swarup generalized the previous construction by
producing an almost invariant set from any $CAT(0)$ cubical complex on which $G$ acts.
We include the formal statement and proof of this result below.
Note that a tree is precisely a $1$-dimensional $CAT(0)$ cubical complex,
and hyperplanes in a tree are midpoints of edges.

\begin{lem}[Lemma 1.17 from \cite{NibloSageevScottSwarup2005}]     
\label{BASICEQUIV}
Let $G$ be any group acting on a cubing $C$.
Let $\mathcal{H}$ be a hyperplane in $C$ with stabilizer $H$,
and suppose that $H$ preserves each of $\mathcal{H^+}$ and $\mathcal{H^-}$.
Then for any vertex $v$, the set $X_v := \{g \in G | gv \in \mathcal{H}^+\}$ is $H$-almost invariant.
Moreover, for any vertices $v$ and $w$, we have $X_v$ is $H$-almost equal to $X_w$.
\end{lem}

In \cite{NibloSageevScottSwarup2005}, the authors assume $G$ is finitely generated,
but their proof does not actually use that assumption.

\begin{proof}
	First, we show that $X$ is $H$-almost invariant.
	Clearly $h X_v = X_v$, for all $h \in H$.
	We also need $X_va$ is $H$-almost equal to $X_v$, for all $a \in G$. We have:
	\begin{align*}
		X_v	&= \{g \in G | gv \in \mathcal{H}^+\} \text{, so that} \\
		X_v a &= \{ ga \in G |gv \in \mathcal{H}^+\} \\
			&= \{ g \in G | g a^{-1}v \in \mathcal{H}^+\}.
	\end{align*}
	To show the symmetric difference of $X_v a$ and $X_v$ is $H$-finite, first we consider one half of the symmetric difference:
	\begin{align*}
		X_v - X_v a 	&= \{ g \in G | gv \in \mathcal{H}^+ \text{ and } g a^{-1} v \notin \mathcal{H}^+\} \\
					&= \{ g \in G | \mathcal{H} \text{ separates } gv \text{ from } g a^{-1} v\} \\
					&= \{g \in G | g^{-1} \mathcal{H} \text{ separates } v \text{ from } a^{-1} v\}.
	\end{align*}
	There are only finitely many (say, $m$) hyperplanes separating $v$ from $a^{-1} v$.
	If $g, g' \in G$ with $g^{-1} \mathcal{H} = {g'}^{-1} \mathcal{H}$,
	then $g' g^{-1} \in Stab(\mathcal{H}) = H$,
	and so $H g'$ and $H g^{-1}$ are actually the same coset.
	We conclude that $X_v - X_v a$ is contained in at most $m$ right cosets $Hg$.
	Similarly, $X_v a - X_v$ is $H$-finite.
	Hence $X_v$ is $H$-almost invariant.

	Second, let $v$ and $w$ be vertices of $C$.
	We need to show that $X_v$ is $H$-almost equal to $X_w$. We have:
	\begin{align*}
		g \in X_v - X_w 		&\iff		gv \in \mathcal{H}^+ \text{ and } gw \notin \mathcal{H}^+ \\
						&\iff		g^{-1} \mathcal{H} \text{ separates } v \text{ from } w.
	\end{align*}
	As in the argument above, the set of all such $g$ is $H$-finite.
	Similarly, $X_w - X_v$ is $H$-finite.
	Hence $X_v$ is $H$-almost equal to $X_w$.
\end{proof}

\subsection{Ultrafilters}

A {\bf partially ordered set with complementation}, or {\bf pocset},
is a partially ordered set $(\Sigma, \leq)$,
equipped with a free involution $*$ on $\Sigma$ behaving like complementation,
i.e $A \leq B$ implies $B^* \leq A^*$.
This terminology was introduced by Sageev and Roller.

\begin{defn}
Let $(\Sigma, \leq)$ be a pocset.
An {\bf ultrafilter} on $(\Sigma, \leq)$ is a subset $V$ of the power set of $\Sigma$ such that both of:
\begin{enumerate}
	\item
	For all $A \in \Sigma$, either $A \in V$ or $A^* \in V$ (but not both), and
	\item
	If $A \in V$ and $A \leq B$, 
	then $B \in \Sigma$.
\end{enumerate}
We say an ultrafilter $V$ satisfies the {\bf descending chain condition (DCC)} if every chain
$A_1 \geq A_2 \geq \ldots$ stabilizes after finitely many steps.
\end{defn}

Note that if $V$ is an ultrafilter on $\Sigma$,
then for any $g \in G$, the translate $gV := \{gA | A \in V \}$ is also an ultrafilter on $\Sigma$.
Also note that $V - \{A\} \cup \{A^*\}$ is an ultrafilter if, and only if, $A$ is a minimal element of $(V, \leq)$.

\subsection{Sageev's Cubing}

In \cite{Sageev1995}, Sageev constructed a cubing $C$ from a finite collection $\{X_j | j = 1, \ldots, n\}$
of \mbox{$H_j$-almost} invariant subsets of a group $G$,
using the partial order of inclusion.
We will now briefly review this construction.
Let $\Sigma := \mycup \limits_{j=1}^n \Sigma(X_j)$.
The vertices of Sageev's cubing are a subset of all ultrafilters on $(\Sigma, \subset)$.
Let $C'$ be the complex with a vertex for each ultrafilter on $(\Sigma, \subset)$,
and an edge connecting each pair of ultrafilters that differ by exactly one complementary pair $(A, A^*)$.
If $V$ is a vertex and $V \cup \{A^*\} - \{A\}$ is also a vertex,
we say the (directed) edge from $V$ to $V \cup \{A^*\} - \{A\}$ \mbox{{\bf exits $A$}.}
See below for the definition of ``basic vertex.''
Define the one-skeleton of $C$ to be the connected component of $C'$ containing all the basic vertices
(equivalently, define the vertices of $C$ to be all ultrafilters satisfying the descending chain condition;
see Lemmas~\ref{BASIC} and \ref{CONNECTED}).
Define higher skeleta of $C$ inductively: whenever you see the boundary of an $n$-cube, attach an $n$-cube.
This is Sageev's cubing.

\begin{defn}[basic vertex]
Let $G$ be a group with a finite collection $\{X_j | j = 1, \ldots, n\}$
of $H_j$-almost invariant subsets.
Let $\Sigma := \mycup \limits_{j=1}^n \Sigma(X_j)$.
Let $g$ be any element of $G$.
Define $V_g$ as follows:
\[
	V_g := \{ A \in \Sigma | g \in A\}.
\]
We call $V_g$ a {\bf basic vertex}. Some authors may refer to basic vertices as basic ultrafilters, principal vertices, or principal ultrafilters.
\end{defn}

\begin{lem}
\label{BASIC}
Let $G$ be a group with any finite collection $\{X_j | j = 1, \ldots, n\}$ of $H_j$-almost invariant subsets.
Let $\Sigma := \mycup \limits_{j=1}^n \Sigma(X_j)$
For each $g \in G$, the basic vertex $V_g$ is an ultrafilter on $(\Sigma, \subset)$ and satisfies DCC.
\end{lem}

\begin{proof}
Fix $g \in G$.
We first show that $V_g$ satisfies conditions 1 and 2 of the definition of ``ultrafilter.''
\begin{enumerate}
	\item
	Let $A, B \in \Sigma$ be arbitrary.
	Either $g \in A$ or $g \in A^*$, so either $A \in V_g$ or $A^* \in V_g$ (but not both).
	\item
	If $A \in V_g$ and $A \subset B$,
	then $g \in B$, so $B \in V_g$.
\end{enumerate}
Hence $V_g$ is an ultrafilter on $(\Sigma, \subset)$.

To show $V_g$ satisfies DCC, take a descending chain $B_1 \supset B_2 \supset \ldots$ in $V_g$.
If the $B_k$ are not all equal to begin with, then without loss of generality (after passing to a subsequence),
$B_1 - B_2$ is nonempty, and
there exists some fixed $j$ such that
$B_k \in \Sigma(X_j)$, for all $k$.
Fix $g_0 \in B_1 - B_2$.
We claim (as proved in Lemma 3.4~of~\cite{Sageev1995} ) that
\[
	\mathcal{B} := \{B \in \Sigma(X_j) |
	g \in B \text{ and } g_0 \notin B, \text{ or } g \notin B \text{ and } g_0 \in B\}
\]
is finite.
Assuming the claim, the chain must stabilize, as each element of the chain (except for $B_1$) is an element of $\mathcal{B}$.

To prove the claim, first note that since $X_j$ is $H_j$-almost invariant, we have
\[
	X_j g^{-1} \overset{H_j-a}{=} X_j g_0^{-1}.
\]
Pick $g^j_1, \ldots, g^j_{r_j}$ such that 
the symmetric difference of $X_j g^{-1}$ and $X_j g_0^{-1}$
is contained in $\coprod \limits_{k = 1}^{r_j} H_j (g^j_k)^{-1}$.
We have:
\begin{align*}
	g' X_j^{(*)} \in \mathcal{B}
	& \iff g' X_j^{(*)} \text{ separates } g \text{ and } g_0\\
	& \iff X_j^{(*)} \text{ separates } (g')^{-1}g \text{ and } (g')^{-1} g_0\\
	& \iff (g')^{-1} \text{ is in the symmetric difference of } X_j g^{-1} \text{ and } X_j g_0^{-1} \\
	& \iff (g')^{-1} \in \coprod \limits_{k = 1}^{r_j} H_j (g^j_k)^{-1}\\
	& \iff g' \in \coprod \limits_{k = 1}^{r_j} g^j_k H_j
\end{align*}
As $H_j$ stabilizes $X_j$, and as there are only finitely many $j$,
it follows that $\mathcal{B}$ is finite.
This completes the proof that $V_g$ satisfies DCC.
\end{proof}

\subsection{Minimal Cubings}

In \cite{NibloSageevScottSwarup2005}, Niblo, Sageev, Scott, and Swarup constructed another cubing $L$
from a finite collection $\{X_j | j = 1, \ldots, n\}$ of $H_j$-almost invariant subsets of a group $G$.
As before, let $\Sigma := \mycup \limits_{j=1}^n \Sigma(X_j)$.
Niblo-Sageev-Scott-Swarup assume that the subsets are already in good position, and use the partial order of almost inclusion.
Their construction requires $G$ and all the $H_j$'s to be finitely generated.
Using the finite generation of $G$ and the $H_j$'s, they constructed analogues of basic vertices,
and defined the vertex set of their cubing $L$ to be everything connected to the basic vertex analogues.
Here, we construct the cubing in the case when the $X_j$'s come from splittings satisfying sandwiching,
and do not require any finite generation assumptions.
As $G$ and the $H_j$'s are not necessarily finitely generated,
we need an alternate way to describe the vertices.
For simplicity, we define the vertices of $L$ to be all ultrafilters on $(\Sigma, \leq)$ satisfying DCC.
This will in fact give us the same vertex set as the cubing in \cite{NibloSageevScottSwarup2005},
in the case when their hypotheses are satisfied; see Lemma~\ref{CONNECTED} for justification.
Later we will need to show that our cubing is always nonempty (see Theorem~\ref{ULTRAFILTER}).

The following lemma proves that the set of vertices connected to a given vertex satisfying DCC
is precisely the set of all vertices satisfying DCC.

\begin{lem}
\label{CONNECTED}
\begin{enumerate}
Let $(\Sigma, \leq)$ be any pocset.
	\item
	Any two vertices (i.e. ultrafilters on $(\Sigma, \leq)$) satisfying DCC
	can be connected via a finite edge path.
	\item
	If a vertex is connected to some vertex satisfying DCC,
	then the vertex satisfies DCC.
\end{enumerate}
\end{lem}

\begin{proof}
\begin{enumerate}
\item
	Assume, for contradiction, that $V$ and $W$ satisfy DCC
	but differ on infinitely many (distinct) elements,
	say $A_1, A_2, \ldots \in V$ and $A_1^*, A_2^*, \ldots \in W$.
	As every element of $\Sigma$ comes from one of finitely many splittings,
	after passing to a subsequence,
	all the $A_k$'s come from a single splitting, and hence are nested.
	Note that as both $V$ and $W$ are ultrafilters,
	we cannot have $A_k \leq A_l^*$.
	Thus, after reordering, we either get an ascending chain in the $A_k$'s
	and descending chain in the $A_k^*$'s, or vice-versa.
	Hence the chain stabilizes, a contradiction to the $A_k$ being distinct.
	\item
	Let $V$ be any vertex satisfying DCC, and $W$ any vertex connected to $V$.
	Then $V$ and $W$ differ by only finitely many complementary pairs $(A, A^*)$.
	Any descending chain in $W$ must have all but finitely many of its elements in $V$,
	so must stabilize (since $V$ satisfies DCC).
\end{enumerate}
\end{proof}

As noted in \cite{NibloSageevScottSwarup2005},
every ultrafilter on $(\Sigma, \leq)$ is also an ultrafilter on $(\Sigma, \subset)$,
and any ultrafilter satisfying DCC with respect to $\leq$ also satisfies DCC with respect to $\subset$,
so that every vertex in $L$ is canonically a vertex in $C$.
We will see that the embedding $L^0 \hookrightarrow C^0$ naturally extends to an embedding $L \hookrightarrow C$.
However, in general, $C$ contains many vertices that are not in $L$.
For example, either all basic vertices are in $L$,
or $L$ contains no basic vertices:

\begin{lem}
Let $G$ be any group with a finite collection $\{X_j | j = 1, \ldots, n\}$
of $H_j$-almost invariant subsets.
Let $\Sigma := \mycup \limits_{j=1}^n \Sigma (X_j)$.
Suppose almost inclusion $\leq$ defines a partial order on $\Sigma$.
Then $V_g$ is an ultrafilter on $(\Sigma, \leq)$ for all $g \in G$
if, and only if, $V_g$ is an ultrafilter on $(\Sigma, \leq)$ for some $g \in G$.
\end{lem}

\begin{proof}
We will use the arguments from the proof of Lemma~\ref{BASIC} show that $V_g$ satisfies DCC with respect to the partial order $\leq$, as follows
(regardless of whether or not $V_g$ actually is an ultrafilter on $(\Sigma, \leq)$).
Suppose we have a descending chain $A_1 \geq A_2 \geq \ldots$ in a basic vertex $V_g$.
Since the $A_i$ come from only finitely many splittings,
after passing to a subchain, we have in fact $A_1 \supset A_2 \supset \ldots$,
which must stabilize, since it is a descending chain in $(V_g, \subset)$ and $(V_g, \subset)$ satisfies DCC.

Now, suppose there exists $g \in G$ such that $V_g$ is not an ultrafilter on $(\Sigma, \leq)$.
Clearly for all $A \in \Sigma$, either $A$ or $A^*$ is in $V_g$.
Hence there must exist $A, B \in \Sigma$ with $A \in V_g$, $A \leq B$, and $B \notin V_g$.
The pair $(A, B)$ prevents $V_g$ from being an ultrafilter.
Now, for any $g' \in G$, the pair $(g' A, g' B)$ prevents $V_{g' g}$ from being an ultrafilter.
\end{proof}

In what follows, we prove that $L$ is nonempty (assuming sandwiching, but not assuming any finite generation).
To create an ultrafilter on $(\Sigma, \leq)$,
we will start with a basic ultrafilter on $\Sigma(X_1) \subset \Sigma$, then extend.
In general, suppose $\Sigma_0 \subset \Sigma$ and $V_0$ is an ultrafilter on $\Sigma_0$.
If we hope to extend $V_0$ to an ultrafilter on all of $\Sigma$,
we must add to $V_0$ all elements $B$ of $\Sigma$ for which $A \leq B$ for some $A \in V_0$.
We call this process taking the closure of $V_0$. More formally:

\begin{defn}
Let $(\Sigma, \leq)$ be any partially ordered set with complementation,
and let $V_0$ be an ultrafilter on some subset $\Sigma_0$ of $\Sigma$.
The {\bf closure of $V_0$} is:
\[
	\overline{V_0} := V_0 \cup \{B \in \Sigma - \Sigma_0 |
	\text{there exists } A \in V_0 \text{ such that } A \leq B\}.
\]
\end{defn}

\begin{lem}
\label{CLOSURE}
Let $(\Sigma, \leq)$ be any partially ordered set with complementation,
and let $V_0$ be an ultrafilter on some subset $\Sigma_0$ of $\Sigma$.
Let $\overline{\Sigma_0}$ denote the set $\{A \in \Sigma | A \in \overline{V_0} \text{ or } A^* \in \overline{V_0}\}$.
Then $\overline{V_0}$, the closure of $V_0$ in $\Sigma$, is an ultrafilter on $(\overline{\Sigma_0}, \leq)$.
\end{lem}

\begin{proof}
To prove that $\overline{V_0}$ is an ultrafilter on $(\overline{\Sigma_0}, \leq)$,
we must show that conditions 1 and 2 of the definition of ``ultrafilter'' are satisfied.
\begin{enumerate}
	\item
	Clearly for all $B \in \overline{\Sigma_0}$, at least one of $B, B^*$ is in $\overline{V_0}$.
	We must show that if $B \in \overline{V_0}$, then $B^* \notin \overline{V_0}$.
	Suppose, for contradiction, that $B \in \overline{V_0}$ and $B^* \in \overline{V_0}$.
	By construction, either $B, B^* \in \Sigma_0$ or $B, B^* \in \overline{\Sigma} - \Sigma_0$, so that
	either $B, B^* \in V_0$ or $B, B^* \in \overline{V_0} -V_0$.
	As $V_0$ is an ultrafilter, it is impossible to have both $B$ and $B^*$ in $V_0$.
	Hence we must have $B, B^* \in \overline{V_0} - V_0$.
	By the definition of ``closure,'' there exist $A, A' \in V_0$ such that $A \leq B$ and $A' \leq B^*$
	(i.e. $B \leq {A'}^*$).
	Transitivity of $\leq$ now implies $A \leq {A'}^*$.
	Since $A \in V_0$, this implies ${A'}^* \in V_0$, a contradiction to $V_0$ being an ultrafilter.
	This completes the proof that we cannot have simultaneously $B \in \overline{V_0}$ and $B^* \in \overline{V_0}$.
	\item
	Assume $B, C \in \overline{\Sigma_0}$ with $B \in \overline{V_0}$ and $B \leq C$.
	We must show that $C \in \overline{V_0}$.
	By the construction of $\overline{V_0}$, there exists $A \in V_0$ with $A \leq B$.
	By transitivity of $\leq$, we have $A \leq C$ (equivalently, $C^* \leq A^*$).
	We break up the rest of the proof into two cases,
	depending on whether $C  \in \Sigma_0$.
	\begin{itemize}
		\item
		If $C \in \Sigma_0$, since $V_0$ is an ultrafilter on $\Sigma_0$,
		either $C$ or $C^*$ must be in $V_0$.
		If $C^* \in V_0$, then $C^* \leq A^*$ implies $A^* \in V_0$, which is impossible since $A \in V_0$.
		Hence we must have $C \in V_0$.
		Since $V_0$ is a subset of $\overline{V_0}$, this implies $C \in \overline{V_0}$.
		\item
		If $C \in \overline{\Sigma_0} - \Sigma_0$,
		then $A \leq C$ implies $C \in \overline{V_0}$.	
	\end{itemize}
	Hence we must have $C \in \overline{V_0}$.
\end{enumerate}
This completes the proof that the closure of an ultrafilter is an ultrafilter.
\end{proof}

The following lengthy theorem proves that the cubing $L$ is non-empty.

\begin{thm}
\label{ULTRAFILTER}
Let $G$ be any group with a finite collection $\{\sigma_j | j = 1, \ldots, n\}$ of pairwise non-isomorphic splittings.
Suppose $\{\sigma_j | j = 1, \ldots, n\}$ satisfies sandwiching.
For each $j$, let $X_j$ be an $H_j$-almost invariant set arising from $\sigma_j$.
Let $\Sigma := \mycup \limits_{j =1}^n \Sigma(X_j)$.
Then there exists an ultrafilter on $(\Sigma, \leq)$ satisfying DCC.
\end{thm}
\begin{proof}
Fix $g \in G$.
We will start with a basic ultrafilter $V_1$ on $\Sigma(X_1)$,
take its closure in $\Sigma$,
and inductively add in part of a basic ultrafilter on each $\Sigma(X_j)$
until we've defined an ultrafilter on all of $\Sigma$.
The ultrafilters produced in all steps are, in order,
$V_1 \subset \overline{V_1} \subset V_2 \subset \overline{V_2} \subset \ldots \subset V_n =: V$.
In Step ja, we add to the ultrafilter $\overline{V_{j-1}}$ all elements $A$ of $\Sigma(X_j)$
such that $g \in A$ and neither $A$ nor $A^*$ was already in the ultrafilter,
to get  the ultrafilter $V_j$.
In Step jb, we take the closure of the ultrafilter from Step ja, to get $\overline{V_j}$.
In the end, we get an ultrafilter $V$ on $\Sigma$.

\begin{itemize}
	\item
	Step 1a:
	Define $V_1$ as a basic ultrafilter on $\Sigma(X_1)$:
	\[
		V_1 := \{A \in \Sigma(X_1) | g \in A\}.
	\]
	Let $\Sigma_1 := \Sigma (X_1).$
	On $\Sigma(X_1)$, the inclusion relation is the same as $\leq$,
	so Lemma~\ref{BASIC} proves that $V_1$ is an ultrafilter on $(\Sigma_1, \leq)$.
	\item
	Step 1b:
	Define $\overline{V_1}$ to be the closure of $V_1$ in $\Sigma$:
	\[
		\overline{V_1} :=
		V_1 \cup \{B \in \Sigma(X_2, X_3, \ldots, X_n) |
		\text{there exists } A \in V_1 \text{ with } A \leq B\}
	\]
	Let $\overline{\Sigma_1} := \{A \in \Sigma | A \in \overline{V_1} \text{ or } A^* \in \overline{V_1}\}$.
	By Lemma~\ref{CLOSURE}, $\overline{V_1}$ is an ultrafilter on $(\overline{\Sigma_1}, \leq)$.
\end{itemize}
Perform the following two steps for $1 < j < n$.
\begin{itemize}
	\item
	Step ja:
	Define $V_j$ to be the union of $\overline{V_{j-1}}$ and part of a basic ultrafilter on $\Sigma(X_j)$:
	\[
		V_j := \overline{V_{j-1}} \cup \{A \in \Sigma(X_j) |
		g \in A \text{ and } A \notin \overline{\Sigma_{j-1}}\}.
	\]
	Let $\Sigma_j := \{A \in \Sigma | A \in V_j \text{ or } A^* \in V_j\}$.
	To prove that $V_j$ is an ultrafilter on $(\Sigma_j, \leq)$,
	we must show that conditions 1 and 2 of the definition of ``ultrafilter'' are satisfied.
	\begin{enumerate}
		\item
		$\Sigma_j$ is defined to be the union of the elements of $V_j$ and their complements.
		We must show that if $A \in V_j$, then $A^* \notin V_j$.
		Suppose, for contradiction, that both $A$ and $A^*$ are elements of $V_j$.
		By construction, either $A, A^* \in \Sigma_j - \overline{\Sigma_{j-1}}$, or $A, A^* \in \overline{\Sigma_{j-1}}$,
		so that either $A, A^* \in V_j - \overline{V_{j-1}}$, or $A, A^* \in \overline{V_{j-1}}$.
		If $A, A^* \in \overline{V_{j-1}}$, this would contradict $\overline{V_{j-1}}$ being an ultrafilter,
		so we must have $A, A^* \in V_j - \overline{V_{j-1}}$.
		This implies $g \in A$ and $g \in A^*$, also a contradiction.
		Hence it is impossible to have simultaneously $A \in V_j$ and $A^* \in V_j$.
		
		\item
		Assume $B, C \in \Sigma_j$ with $B \in V_j$ and $B \leq C$.
		We must show that $C \in V_j$.
		
		We break up the proof that $C \in V_j$ into two cases,
		depending on whether $B$ was added to the ultrafilter in Step ja or a previous step.
		
		\begin{itemize}
			\item
			Suppose $B$ was added in Step ja, i.e. $B \in V_j - \overline{V_{j-1}}$.
			Note that $B \in \Sigma(X_j)$.
			Either $C \in \overline{\Sigma_{j-1}}$ or $C \in \Sigma_j - \overline{\Sigma{j-1}}$.
			We want to show that $C \in V_j$.
			\mbox{If $C \in \overline{\Sigma_{j-1}}$} and $C \notin V_j$,
			then we must have $C^* \in \overline{V_j}$.
			Then $C^* \leq B^*$ would imply $B^*$ was added to the ultrafilter by Step (j-1)b,
			so it would be impossible to have $B \in V_j - \overline{V_{j-1}}$.
			If instead $C \in \Sigma_j - \overline{\Sigma{j-1}}$, then $C \in \Sigma(X_j)$.
			As $B$ and $C$ are both elements of $\Sigma(X_j)$,
			having $B \leq C$ implies $B \subset C$.
			Hence $g \in B$ implies $g \in C$, so that $C \in V_j$.
			\item
			Suppose $B$ was added in a previous step, i.e. $B \in \overline{V_{j-1}}$.
			Then by construction of $ \overline{V_{j-1}}$,
			there exists $A \in V_{j-1}$ with $A \leq B$.
			By transitivity of $\leq$, we have $A \leq C$.
			It follows that $C \in \overline{V_{j-1}}$.
			As $\overline{V_{j-1}} \subset V_j$, we have $C \in V_j$.
		\end{itemize}
		In any case, we conclude $C \in V_j$.
	\end{enumerate}
	Hence $V_j$ is an ultrafilter on $\Sigma_j$.
	\item
	Step jb:
	Define $\overline{V_j}$ to be the closure of $V_j$ in $\Sigma$:
	\[
		\overline{V_j} := V_j \cup \{B \in \Sigma(X_{j+1}, X_{j+2}, \ldots, X_n) |
		\text{there exists } A \in V_j \text{ with } A \leq B\}.
	\]
	Let $\overline{\Sigma_j} := \{A \in \Sigma | A \in \overline{V_j} \text{ or } A^* \in \overline{V_j}\}$.
	By Lemma~\ref{CLOSURE}, $\overline{V_j}$ is an ultrafilter on $(\overline{\Sigma_j}, \leq)$.
\end{itemize}
Perform one last step to define an ultrafilter on all of $\Sigma$.
\begin{itemize}
	\item
	Step na (this is just Step ja with $j = n$)
\end{itemize}
Note that $\Sigma_n = \Sigma$, so there is no need for Step nb.
Let $V := V_n$.

We have successfully defined an ultrafilter $V$ on $(\Sigma, \leq)$. Next we prove that $V$ satisfies DCC.
To make the proof less cumbersome, I will write {\bf WLOG} to denote {\bf without loss of generality}.

Suppose we have an (infinite) descending chain $B_1 \geq B_2 \geq \ldots$ in $V$.
We will obtain $A_k \leq B_k$, show that the $A_k$'s stabilize,
and then show that the $B_k$'s stabilize.
\begin{enumerate}
	\item
	Since we're dealing with only $n$ splittings,
	WLOG (after passing to a subchain of $(B_k)_{1 \leq k}$)
	there exists a fixed $j \in \{1, 2, \ldots, n\}$ such that $B_k \in \Sigma(X_j)$.
	\item
	If there exists an infinite subchain of the $B_k$'s that were added to $V$ in Stepja,
	then since $g$ is in each element of the subchain,
	the proof of Lemma~\ref{BASIC} shows that the subchain must stabilize,
	so that the original chain stabilizes, and we're done.
	Otherwise, WLOG (after passing to a subchain of $(B_k)_{1 \leq k}$),
	each of the $B_k$'s was added to $V$ in a type ``b'' Step (before Stepja).
	Recall that $j$ was fixed in the previous step, and $B_k \in \Sigma(X_j)$, for all $k$.
	\item
	For each $B_k$, since $B_k$ was added in a type ``b'' Step before Step ja,
	there exists $A_k \in  V_{j-1}$ such that $A_k \leq B_k$.
	\item
	WLOG (after possibly replacing $A_k$ by something less than $A_k$),
	each $A_k$ was added in a type ``a'' Step, so that $g \in A_k$.
	\item
	WLOG (after passing to a subchain of $(A_k \leq B_k)_{1 \leq k}$),
	there exists a fixed $j'$ (for some $1 \leq j' \leq j-1$) such that $A_k \in \Sigma(X_{j'})$, for all $k$.
	In particular, all the $A_k$'s are nested.
	\item
	We cannot have $A_k \subset A_l^*$ or $A_k^* \subset A_l$
	(since all the $A$'s belong to the ultrafilter $V$),
	hence for all $k \neq l$, we must have $A_k \subset A_l$ or $A_k \supset A_l$.
	\item
	In this step, we show that the $A_k$'s stabilize.
	If there is an infinite subchain of the $A_k$'s such that each is contained in the next,
	then WLOG (after replacing each $A_k$ in the subchain by the first one)
	all the $A_k$'s in that subchain are equal, so move on to the next step of the proof.
	Otherwise, WLOG (after passing to a subchain of $(A_k \leq B_k)_{1 \leq k}$)
	we have $A_k \supset A_l$, for all $k < l$.
	Since $g \in A_k$ for all $k$, and since all of the $A_k$'s are in $\Sigma(X_{j'})$,
	the $A_k$'s must stabilize after finitely many steps (by Lemma~\ref{BASIC}).
	So WLOG (after passing to a subchain of $(A_k \leq B_k)_{1 \leq k}$),
	all the $A_k$'s are identical.
	\item
	Recall that $B_k \geq A_1$, for all $k$.
	We now have $B_1 \geq B_2 \geq \ldots \geq A_1$.
	But since \mbox{$\{\sigma_j | 1 \leq j \leq n\}$} satisfies sandwiching,
	this contradicts interval finiteness (see Proposition~\ref{INTERVALS}), unless the $B_k$'s stabilize.
	Hence the $B_k$'s stabilize.
\end{enumerate}
This completes the proof that $V$ satisfies DCC.
In particular, we have shown there exists an ultrafilter on $(\Sigma, \leq)$ satisfying DCC.
\end{proof}

We have shown that $L$ is nonempty.
Furthermore, the proof of Theorem~\ref{ULTRAFILTER} shows that each $A \in \Sigma$ is in some vertex $L$:
pick any $g \in A$, reorder the splittings (and their associated $X_j$)
such that $A \in X_1$, and construct $V$ as in the proof of Theorem~\ref{ULTRAFILTER}.

\subsection{Putting the $X_j$'s in ``Very Good Position''}

Let $G$ be any group.
Take any finite collection $\{\sigma_1, \ldots, \sigma_n\}$ of pairwise non-isomorphic splittings of $G$ satisfying sandwiching.
Let $X_j$ be a standard $H_j$-almost invariant set arising from $\sigma_j$,
and let $\Sigma = \{ gX_j, gX_j^* | g \in G, j = 1, \ldots, n \}$.
We have shown that $\Sigma$ is in good position,
i.e. if two corners of $(A,B)$ are small, then (at least) one is empty.
This allowed us to define the partial order $\leq$ on $\Sigma$.
In this section, we show how to find $X_j' \overset{H_j-a}{=} X_j$
such that \mbox{$\Sigma ' = \{ gX'_j, {gX'}_j^* | g \in G, j = 1, \ldots, n \}$} is in very good position,
i.e. a corner of the pair $(A', B')$ is small if, and only if, the corner is empty.
This result was previously proved in \cite{NibloSageevScottSwarup2005} for a finite collection of almost invariant sets
over finitely generated subgroups of a finitely generated group.

\begin{thm}
\label{VGP}
Let $G$ be any group with a finite collection $\{\sigma_j | j = 1, \ldots, n\}$ of pairwise non-isomorphic splittings.
Suppose $\{\sigma_j | j = 1, \ldots, n\}$ satisfies sandwiching.
For each $j$, let $X_j$ be a standard $H_j$-almost invariant set arising from $\sigma_j$.
Then there exist $X'_j \overset{H_j-a}{=} X_j$,
such that $\Sigma' := \mycup \limits_{j =1}^n \Sigma(X'_j)$
is in very good position.
\end{thm}

In preparation for proving the theorem, we take a look at how hyperplanes in the cubings $C$ and $L$ compare to each other.
For each $j$, let $\mathcal{H}_j$ be the hyperplane in $C$
determined by the equivalence class of edges in $C$ exiting $X_j$.
(Or, equivalently, the class of edges equivalent to any given edge exiting $X_j$,
with the equivalence relation generated by square-equivalence.)
Define the halfspace $\mathcal{H}_j^+$ of $C$ by:
\[
	\mathcal{H}_j^+ = \{V \in C^{(0)} | X_j \in V\}.
\]
For any vertex $v \in C$, define $(X_j)_v$ by:
\[
	(X_j)_v := \{g \in G | gv \in \mathcal{H}_j^+\}.
\]
Note that $(X_j)_{V_e} = X_j$,
where $V_e$ is the basic ultrafilter on $(\Sigma, \subset)$ consisting of all elements of $\Sigma$ containing the identity.
Let $\mathcal{K}$ be the hyperplane in $L$ determined by the equivalence class of edges exiting $X_j$.
Define the halfspace $\mathcal{K}_j^+$ of $L$ by:
\[
	\mathcal{K}_j^+ := \{W \in L^{(0)} | X_j \in W\}.
\]

Recall the canonical embedding $L_0 \hookrightarrow C_0$,
in which we view any vertex in $L$, i.e. an ultrafilter $W \subset \Sigma$ on $(\Sigma, \leq)$ satisfying DCC,
as an ultrafilter on $(\Sigma, \subset)$.
Now we will see how this extends to an embedding $L \hookrightarrow C$.
If two edges in $L$ are opposite sides of a square in $C$
then all four vertices of the square in $C$ are in $L$,
so that the other two edges of the square in $C$ are also in $L$.
Hence two edges in $L$ are on the opposite sides of a square in $L$ if, and only if,
they are on opposite sides of a square in $C$.
It follows that $\mathcal{H}_j \cap L = \mathcal{K}_j$ and
$\mathcal{H}_j^+ \cap L = \mathcal{K}_j^+$.

Now we can use the cubing $L$ to put the $X_j$'s in very good position.

\begin{lem}
Fix a vertex $w \in L \subset C$, and define $X'_j := \{g \in G | gw \in \mathcal{K}_j^+ \}$.
Then each $X'_j$ is $H_j$-almost invariant, the collection $\Sigma' := \mycup \limits_{j =1}^n \Sigma(X'_j)$ is in very good position,
and $X_j \mapsto X'_j$ induces a $G$-equivariant isomorphism from $(\Sigma, \leq)$ to $(\Sigma', \subset)$.
\end{lem}

\begin{proof}
Viewing $w$ as a vertex in $C$, we have $(X_j)_w = \{ g \in G | gw \in \mathcal{H}_j^+ \}$.
As $L \hookrightarrow C$ is \mbox{$G$-equivariant,}
it follows that $X'_j = (X_j)_w$.
Now, applying Lemma~\ref{BASICEQUIV},
$X'_j = (X_j)_w$ is $H_j$-almost invariant and $H_j$-almost equal to $(X_j)_{V_e} = X_j$.
Clearly $\Sigma'$ is in very good position.

Next we show, as proved in Lemma 4.1 of \cite{NibloSageevScottSwarup2005},
that $X_j \mapsto X'_j$ induces a $G$-equivariant isomorphism from $(\Sigma, \leq)$ to $(\Sigma', \subset)$.
We are assuming that none of the $X_j$ come from isomorphic splittings.
$w$ is an ultrafilter on $(\Sigma, \leq)$, so for all $g \in G$,
$gw$ is also an ultrafilter on $(\Sigma, \leq)$.
\mbox{If $A, B \in \Sigma$,} let $A', B'$ denote the images in $\Sigma'$ of $A$ and $B$
by the map that sends $X_j \mapsto X'_j$.

Suppose $A \leq B$, i.e. $A \cap B^*$ is empty or the only small corner of the pair $(A,B)$.
Note that since $X_j \overset{H_j-a}{=} X'_j$,
we have $A \overset{Stab(A)-a}{=} A'$ and $B \overset{Stab(B)-a}{=} B'$.
Hence a corner of the pair $(A,B)$ is small if, and only if, the corresponding corner of the pair $(A', B')$ is small.
If $A \cap B^*$ is the only small corner of the pair $(A,B)$,
then since $A'$ and $B'$ are nested,
we must have $A' \subset B'$, as desired.
If the pair $(A, B)$ has two small corners, then $A \subset B$.
To see that $A' \subset B'$, simply note that
since $X'_j = \{ g \in G | gw \in \mathcal{H}_j^+\}$,
we have $A' = \{g \in G | A \in gw \}$,
and $B' = \{ g \in G | B \in gw \}$.
In either case, we conclude that $A' \subset B'$.

Conversely, suppose that $A' \subset B'$.
We need to show that $A \leq B$.
Since $\Sigma$ is in good position, it follows that $A \leq B$ or $B \leq A$.
The above paragraph shows that $A \leq B$.
Hence $A \leq B$ if, and only if, $A' \subset B'$.
\end{proof}

This completes the proof of Theorem~\ref{VGP}.

\section{Existence of Algebraic Regular Neighborhoods}
\label{AlgebraicRegularNeighborhoodsE}

Start with a finite collection of splittings $\{\sigma_1, \ldots, \sigma_n\}$ over subgroups $H_j$ of a group $G$.
Let $X_j$ be an $H_j$-almost-invarint set arising $\sigma_j$.
Suppose $\{X_1, \ldots, X_n\}$ satisfies sandwiching
(see Definition~\ref{DEFNSANDWICHING}).
Also assume no two of the $\sigma_j$'s are isomorphic to each other.
We will construct a bipartite $G$-tree $T(X_1, \ldots, X_n)$,
and show $T$ is an algebraic regular neighborhood of $\{\sigma_1, \ldots, \sigma_n\}$ (see Definition~\ref{defnARN}).

Let $\Sigma$ denote the set of all translates of all the $X_j$ and their complements.
Since no two $\sigma_j$'s are isomorphic to each other, $\leq$ defines a partial order on $\Sigma$
(see Corollary~\ref{PARTIALORDER}).
Construct the cubing $L$ from $(\Sigma, \leq)$,
as detailed in Section~\ref{VeryGoodPosition}.
We will construct a bipartite tree from the $L$.

Define a ``cross connected'' relation on $\Sigma$ as follows:
\begin{equation*}
	A \text{ is cross-connected to } B \iff
	\begin{array}{l}
		A \text{ is equal to } B \text{ or } B^* \text{, or} \\
		\text{there exists some } m \geq 0
		\text{ and a sequence } (A, B_1, \ldots, B_m, B)\\
		\text{such that } A \text{ crosses } B_1, B_1 \text{ crosses } B_2, \ldots,\\
		B_{m-1} \text{ crosses } B_m,
		\text{ and } B_m \text{ crosses } B.\\
	\end{array}
\end{equation*}
This defines an equivalence relation.
Call each equivalence class a {\bf cross connected component} ($ccc$).

One can easily see $ccc$'s of $\Sigma$ from looking at the cubing $L$.
Each (directed) hyperplane in the cubing corresponds to a unique element of $\Sigma$.
If we remove all cut vertices from the cubing,
we are left with a disjoint collection of components,
where each component is a subcubing with some vertices missing.
Each $ccc$ has all its hyperplanes contained in a single component.
Moreover, since the components have no cut vertices,
each component's hyperplanes come from only one $ccc$.
So we have a natural bijective correspondence between components and the $ccc$'s.

We introduce some basic notation.
View the cubing as a disjoint union of subcubings which are glued together at cut vertices.
Let $CUT$ denote the set of cut vertices.
Let $SUB$ denote the set of (disjoint) subcubings.
Note that we have a bijective correspondence between $SUB$ and the $ccc$'s of $\Sigma$.
For a given subcubing $\alpha \in SUB$, define the {\bf corner vertices of $\alpha$} to be
the vertices of $\alpha$ that are glued to cut vertices.
Let $CRN(\alpha)$ denote the set of corner vertices of $\alpha$.

For each subcubing $\alpha \in SUB$,
create a tree whose vertices are $CRN(\alpha)$ plus a central vertex,
and an edge connecting each element of $CRN(\alpha)$ to the central vertex.
Call this tree~$\bigstar_\alpha$.

Glue the $\bigstar_\alpha$'s together by, for each element of $CUT$,
identifying all corner vertices of all the~$\bigstar_\alpha$'s which came from that element of $CUT$.
Color the equivalence classes of corner vertices as $V_1$-vertices,
and color all the central vertices as $V_0$-vertices.

The result is a bipartite $G$-tree.
Let $T(X_1, \ldots, X_n)$ denote this tree.

\begin{thm}[Existence of algebraic regular neighborhoods]
\label{thmARNE}
Let $G$ be any group with a finite collection $\{\sigma_j | j = 1, \ldots, n\}$ of pairwise non-isomorphic splittings.
Suppose $\{\sigma_j | j = 1, \ldots, n\}$ satisfies sandwiching.
For each $j$, let $X_j$ be a standard $H_j$-almost invariant set arising from $\sigma_j$.
Then $T := T(X_1, \ldots, X_n)$ is algebraic regular neighborhood of $\{\sigma_j | j = 1, \ldots, n\}$.
\end{thm}

To prove the theorem, we need to show that $T$ satisfies the five conditions of the definition of
algebraic regular neighborhood (from Page \pageref{DEFNofARN}).

\begin{lem}[First condition]
Each $\sigma_j$ is enclosed by some $V_0$-vertex orbit in $T$,
and each $V_0$-vertex orbit encloses some $\sigma_j$.
\end{lem}

\begin{proof}
Fix $j$. We will use the original cubing $L$ to construct a particular refinement of $T$.
Recall that $CUT$ denotes the cut vertex set of the original cubing,
$SUB$ denotes the set of subcubings,
and $CRN(\alpha)$ denotes the set of corner vertices of $\alpha \in SUB$.
For each subcubing $\alpha \in SUB$ not containing any $X_j$-hyperplanes,
define $\bigstar_\alpha$ as above.

For each subcubing $\alpha \in SUB$ that contains $\Sigma(X_j)$-hyperplanes.
Let $\#_\alpha$ denote the dual tree to the $\Sigma(X_j)$-hyperplanes in $\alpha$.
For each element of $CRN(\alpha)$, make a vertex and attach it to $\#_\alpha$ with an edge.
Specifically, attach the edge to the vertex of $\#_\alpha$
that corresponds to the component of $\alpha~-~(\Sigma(X_j) \text{-hyperplanes})$ containing the corner vertex.
Call this tree $\bigstar'_\alpha$.

Glue the $\bigstar'_\alpha$'s (for $ccc$'s containing $\Sigma(X_j)$-hyperplanes)
and $\bigstar_\alpha$'s (for $ccc$'s not containing $\Sigma(X_j)$-hyperplanes)
together by, for each element of $CUT$,
identifying all corner vertices which came from that element of $CUT$.
Color the equivalence classes of corner vertices as $V_1$-vertices,
and color all the other vertices as $V_0$-vertices.

This new tree maps naturally to $T$ by collapsing the new edge orbit.
On the other hand, the new tree maps to a tree for $\sigma$
by collapsing all edges except for the new edge orbit.
\end{proof}

\begin{lem}[Second condition]
\label{ARN2}
If $\sigma$ is a splitting of $G$ over $H$,
where for all $j \in J$
$\sigma$ is sandwiched by $\sigma_j$
and $i(\sigma, \sigma_j) = 0$,
then $\sigma$ is enclosed by some $V_1$-vertex orbit in $T$.
\end{lem}

\begin{proof}
Let $X$ be a standard $H$-almost-invariant set arising from $\sigma$,
where $X$ is sandwiched by $X_j$
and $i(\sigma, \sigma_j)=0$, for all $1 \leq j \leq n$.

Construct a new cubing from $\Sigma(X_1, \ldots, X_n, X)$ (using the partial order $\leq$).
Since $X$ does not cross any element of $\mycup \limits_{j =1}^n \Sigma(X_j)$,
each new subcubing $\alpha$ in $SUB$ is simply an edge,
and $\bigstar_\alpha$ consists of two $V_1$ vertices connected to a $V_0$-vertex.
Construct the new tree $T(X_1, \ldots, X_n, X)$.
The new tree projects naturally to a tree for $\sigma$,
by collapsing each old $\bigstar_\alpha$ to a point, and forgetting the new $V_0$-vertices.
On the other hand, the new tree naturally projects to the old tree,
with each new $\bigstar_\alpha$ being collapsed to a single $V_1$-vertex.
Hence $\sigma$ is enclosed by a $V_1$-vertex orbit.
\end{proof}

\begin{lem}[Third condition] 
$T$ is a minimal $G$-tree.
\end{lem}

\begin{proof}
Let $T_0$ be the minimal sub-$G$-tree of $T$ (or any fixed vertex, if $G$ fixes a vertex of $T$).

If $T_0$ has no $V_0$-vertices, then since $T$ is bipartite, $T_0$ must consist of a single $V_1$-vertex which is fixed by $G$.
Let $V$ denote a $V_0$-vertex adjacent to the fixed $V_1$-vertex.
Since the orbit of $V$ encloses  $\sigma_j$ for some $j$,
and since $\Sigma(X_j)$ has infinite chains,
$V$ satisfying DCC implies that
there exists a translate of $V$ not adjacent to the fixed $V_1$-vertex.
This is impossible.
Hence $T_0$ must contain a $V_0$-vertex.

Take any $V_0$-vertex $V'$ in $T_0$,
and pick $j \in J$ such that $\sigma_j$ is enclosed by the orbit of $V'$.
\mbox{If $T_0 \neq T$,} then we can find a $V_0$-vertex $V''$ in $T - T_0$.
Pick $k \in J$ such that $\sigma_k$ is enclosed by the orbit of $V''$.
As $\sigma_j$ is sandwiched by $\sigma_k$,
there exists a translate of $V'$ that is not in $T_0$.
This is impossible, as $T_0$ is $G$-invariant.
Hence we must have $T_0 = T$.
This completes the proof that $T$ is a minimal $G$-tree.
\end{proof}

\begin{lem}[Fourth condition]
There exists a bijection
\[
	f: 
	\{j \in J | \sigma_j \text{ is isolated}\}
	{\rightarrow} G\text{-orbits of isolated } V_0 \text{-vertices of } T
\]
such that $f(j)$ encloses $\sigma_j$.
\end{lem}

\begin{proof}
Each isolated $V_0$-vertex corresponds to a subcubing $\alpha \in SUB$
consisting of exactly one edge, or equivalently, exactly one hyperplane.
This hyperplane corresponds to a unique pair $\{A, A^*\} \subset \Sigma$.
\end{proof}

\begin{lem}[Fifth condition]
Every non-isolated $V_0$-vertex orbit in $T$ encloses some non-isolated $\sigma_j$.
\end{lem}

\begin{proof}
Any non-isolated $V_0$-vertex corresponds to a subcubing $\alpha \in SUB$
containing at least two hyperplanes that cross each other.
\end{proof}

This completes the proof of Theorem~\ref{thmARNE}.

\section{Uniqueness of Algebraic Regular Neighborhoods}
\label{AlgebraicRegularNeighborhoodsU}

We prove uniqueness of algebraic regular neighborhoods for an arbitrary collection of splittings of $G$ satisfying sandwiching.

\begin{thm}[Uniqueness of algebraic regular neighborhoods]
\label{UNIQUENESSOFNBHD}
Let $G$ be any group with any collection $\{\sigma_j | j \in J\}$ of pairwise non-isomorphic splittings.
Suppose $\{\sigma_j | j \in J\}$ satisfies sandwiching,
and that $T_1$ and $T_2$ are algebraic regular neighborhoods of $\{ \sigma_j | j \in J\}$.
Then there exists a $G$-equivariant, color preserving isomorphism from $T_1$ to $T_2$.
\end{thm}

The proof of Theorem~\ref{UNIQUENESSOFNBHD} is laid out in this section.
We will use the same strategy Scott and Swarup used to prove Theorem 6.7 of \cite{ScottSwarup2003}.
Namely, insert an edge orbit in $T_1$ for each edge splitting of $T_2$ that is not already an edge splitting of $T$,
and vice versa.
Then we will show a contradiction if we actually had to insert any edge orbits.
To ``insert edge orbits'' in $T_1$ or $T_2$,
we need to know that the edge splittings in $T_1$ and $T_2$ are compatible with the edge splittings to be inserted.

\begin{lem}
\label{JOINTLYCOMPATIBLE}
Suppose $\{\sigma_k | k \in K\}$ and $\{\sigma_l | l \in L\}$ are collections of splittings of $G$,
such that their union satisfies sandwiching.
Assume $\{\sigma_k | k \in K\}$ and $\{\sigma_l | l \in L\}$ are each compatible,
and that $i(\sigma_k, \sigma_l) = 0$ for all $k \in K$ and $l \in L$.
Then $\{\sigma_k | k \in K \cup L\}$ is compatible.
\end{lem}

\begin{proof}
First we pick a standard almost invariant set for each isomorphism class of splittings.
For each $k \in K \cup L$,
let $X_\alpha$ be a standard almost-invariant set arising from $\sigma_k$.
Without loss of generality, if $\sigma_k$ is isomorphic to $\sigma_l$,
then $X_k = X_l$ (as subsets of $G$).
Let $\Sigma$ denote the collection of all translates of the standard almost invariant sets and their complements:
\[
	\Sigma := \mycup \limits_{k \in K \cup L} \Sigma(X_k).
\]

The proof of Theorem~\ref{COMPATIBLE} directly carries through,
provided we can prove interval finiteness.
We need to show that for all $A, B \in \Sigma$, there are only finitely many $C \in \Sigma$ with $A \leq C \leq B$.

If $A$ and $B$ both are in $\mycup \limits_{k \in K } \Sigma(X_k)$
or both are in $\mycup \limits_{l \in L} \Sigma(X_l)$
then we're done, since $\{\sigma_k | k \in K\}$ and $\{\sigma_l | l \in L\}$ are each compatible.

If $A \in \Sigma(X_k)$ for some $k \in K$, and $B \in \Sigma(X_l)$ for some $l \in L$,
then  by the sandwiching assumption, we can find $B' \in \Sigma(X_l)$
such that $B' \leq A$, and $A' \in \Sigma(X_k)$ such that $B \leq A'$.
There are only finitely many $C \in \mycup \limits_{k \in K } \Sigma(X_k)$ with $A \leq C \leq A'$, and only finitely many $C \in \mycup \limits_{l \in L } \Sigma(X_l)$ with $B' \leq C \leq B$, and so there are only finitely many $C \in \Sigma$ with $A \leq C \leq B$.

If $B$ comes from a splitting $\sigma_k$ for some $k \in K$ and
$A$ comes from a splitting $\sigma_l$ for some $l \in L$,
a similar argument shows there are only finitely many $C \in \Sigma$ with $A \leq C \leq B$.

Apply Dunwoody's theorem to get a tree (see Section~\ref{SUBdunwoody}).
For each edge orbit $Ge$, let $n_e$ denote
the number of splittings in $\{\sigma_k | k \in K \cup L\}$
that are isomorphic to the edge splitting for $e$,
and subdivide each edge in $Ge$ into an interval of $n_e$ edges.
\end{proof}

To apply the above lemma, we need to know that each splitting of an algebraic regular neighborhood
is sandwiched by each $\sigma_j$.

\begin{lem}
\label{EDGESPLITTING}
Let $G$ be any group with any collection $\{\sigma_j | j \in J\}$ of pairwise non-isomorphic splittings.
Suppose $\{\sigma_j | j \in J\}$ satisfies sandwiching,
and let $T$ be an algebraic regular neighborhood of $\{\sigma_j | j \in J \}$.
Every edge splitting of $T$ is sandwiched by $\{\sigma_j | j \in J \}$.
\end{lem}

\begin{proof}
Assume there exists some edge $e$ of $T$
and some $j \in J$
such that splitting from $e$
is not sandwiched by $\sigma_j$.
Let $\sigma$ denote the splitting from $e$.

Let $V$ be some $V_0$-vertex of $T$ whose orbit encloses $\sigma_j$.
The convex hull of all translates of $V$ is a $G$-invariant subtree of $T$.
The assumption that $\sigma$ is not sandwiched by $\sigma_j$
implies that all translates of $V$ lie on one side of $e$,
so that $e$ is not in the convex hull of all translates of $V$.
This implies that $T$ is not a minimal $G$-tree, a contradiction to $T$ being an algebraic regular neighborhood.
\end{proof}

Now we present the proof of Theorem~\ref{UNIQUENESSOFNBHD}.

\begin{proof}
First we prove the theorem in the case that no $\sigma_j$ is an isolated splitting,
i.e. no $\sigma_j$ has intersection number zero with every other splitting in the collection.
We will show that the edge splittings of $T_1$ and $T_2$ are isomorphic.
Then by a uniqueness result (see Corollary~\ref{UNIQUENESSTREE}),
$T_1$ and $T_2$ are $G$-isomorphic.
Assume (for contradiction) that $T_1$ and $T_2$ have different edge splittings.
By ``different edge splittings,'' we mean that $T_1$ (or $T_2$)
has an edge splitting not isomorphic to any edge splitting in $T_2$ (or $T_1$),
or that $T_1$ (or $T_2$) has strictly more edge orbits than $T_2$ (or $T_1$)
yielding splittings in a given isomorphism class.

Let $\sigma$ be some edge splitting of $T_1$, call it $\sigma$, that is not in $T_2$ (in the sense described above).
As each $\sigma_j$ is enclosed by some $V_0$-vertex of $T_1$,
each $\sigma_j$ has intersection number zero with $\sigma$.
Moreover, by Lemma~\ref{EDGESPLITTING}, $\sigma$ is sandwiched by each $\Sigma(X_j)$.
By condition number $2$ of the definition of algebraic regular neighborhood,
$\sigma$ is enclosed by some $V_1$-vertex of $T_2$,
so that we can refine $T_2$ by adding one edge orbit which represents $\sigma$.

By Lemma~\ref{JOINTLYCOMPATIBLE}, we can apply the above procedure simultaneously
for all edge splittings of $T_1$ that are not in $T_2$.
Let the tree $T_{21}$ denote a tree obtained from $T_2$
by splitting at $V_1$-vertices for each edge splitting of $T_1$ that was not already in $T_2$.
When splitting at a $V_1$-vertex, color both endpoints of the new edge as $V_1$-vertices.
Define $T_{12}$ similarly.
$T_{12}$ and $T_{21}$ may have infinitely many edge orbits, but by Corollary~\ref{UNIQUENESSTREE},
$T_{12}$ and $T_{21}$ are isomorphic $G$-trees.

If $T_{21}$ has an edge $e$ not in $T_2$, then when adding in $e$,
we would have split $T_2$ at a $V_1$-vertex.
Under the isomorphism from $T_{21}$ to $T_{12}$, the edge $e$ must map to an original edge of $T_1$,
so the isomorphism must identify a $V_1$-vertex of $T_{21}$ with an original $V_0$-vertex of $T_1$.
Pick some splitting $\sigma_k$ enclosed by that $V_0$-vertex.
By the isomorphism, $\sigma_k$ is enclosed by a $V_1$-vertex of $T_2$,
and hence has intersection number zero with every splitting in $\{\sigma_j | j \in J \}$,
so that $\sigma_k$ is an isolated splitting.
This contradicts the assumption that none of the splittings in $\{\sigma_j | j \in J \}$ are isolated.
Hence no edges were added to $T_2$, i.e. $T_{21} = T_2$.

A similar argument shows that $T_{12} = T_1$.
Hence the isomorphism between $T_{21}$ to $T_{12}$ is actually an isomorphism between $T_2$ and $T_1$.
If the isomorphism did not preserve color, then as in the above paragraph,
the isomorphism would identify a $V_0$-vertex of one tree with a $V_1$ vertex of another,
and hence one of the $\sigma_j$'s would be isolated.
This completes the proof of uniqueness of algebraic regular neighborhoods,
in the case where no $\sigma_j$ is isolated.

Second, we prove the theorem in the case where $\{\sigma_j | j \in J \}$ has some isolated splittings.
For each $V_0$-vertex in an orbit corresponding to an isolated $\sigma_j$ in the definition of algebraic regular neighborhood, forget the vertex.
This leaves an edge bounded by two $V_1$-vertices and yielding a splitting isomorphic to $\sigma_j$.
Let $T_1'$ denote the resulting tree. Define $T_2'$ similarly.

If all the $\sigma_j$'s are isolated,
then no $V_0$-vertices remain,
so $T_1'$ and $T_2'$ are compatibility trees for $\{\sigma_j | j \in J\}$.
Then Corollary~\ref{UNIQUENESSTREE} proves that $T_1'$ and $T_2'$ are $G$-isomorphic.

If not all of the $\sigma_j$'s are isolated, consider each edge splitting edge splitting in $T'_1$ that is not in $T'_2$.
Without loss of generality, we can take each such edge orbit to consist of edges where one endpoint is $V_0$ and the other is $V_1$
(as opposed to edges bounded by two $V_1$-vertices, resulting from a forgotten $V_0$-vertex).
Now apply the above procedure to $T_2'$ and $T_1'$ to obtain $T_{21}'$ and $T_{12}'$,
and an isomorphism from $T_{21}'$ to $T_{12}'$.
If $T_{21}' \neq T_2'$ or $T_{12}' \neq T_1'$, then the isomorphism from $T_{21}'$ to $T_{12}'$
must identify a $V_0$ and a $V_1$-vertex.
This is impossible, as the non-forgotten $V_0$-vertices are not isolated.
Similarly, the isomorphism from $T_{21}'$ to $T_{12}'$ must preserve color.
Hence we get a $G$-equivariant, color preserving isomorphism from $T_2'$ to $T_1'$.

To get an isomorphism from $T_2$ to $T_1$, add a $V_0$-vertex in the middle of every edge
bounded by two $V_1$-vertices.
This completes the proof of uniqueness of algebraic regular neighborhoods.
\end{proof}

\section{Mixed Almost Invariant Sets when G is Finitely Generated}
\label{VGPforGFG}

The cubing construction outlined in Section~\ref{VeryGoodPosition}
can also be applied to a finite collection of \mbox{$H_i$-almost} invariant subsets of $G$ together with a finite collection of standard $K_j$-almost invariant sets arising from splittings of $G$, provided $G$ and all the $H_i$ are finitely generated, and the combined family of almost-invariant sets satisfy sandwiching.
This can be used to put a ``mixed'' family of almost invariant sets,
where some come from splittings
and others have finitely generated stabilizers,
in very good position.

\begin{thm}
\label{MIXING}
Let $G$ be a finitely generated group with any finite collection $\{X_i | i = 1, \ldots, m\}$ of $H_i$-almost invariant subsets,
where each $H_i$ is finitely generated, each $X_i$ is nontrivial,
and $\mycup \limits_{i =1}^m \Sigma(X_i)$ is in very good position.
Let $\{\sigma_j | j =1, \ldots, n\}$ be any finite collection of pairwise non-isomorphic splittings of $G$.
For each $j$, let $Y_j$ be a $K_j$-almost invariant set arising from $\sigma_j$,
where each $K_j$ is not finitely generated.
Assume that $\{X_1, \ldots, X_m, Y_1, \ldots, Y_n\}$ satisfies sandwiching.
Let $\Sigma := \mycup \limits_{i =1}^m \Sigma(X_i) \cup \mycup \limits_{j =1}^n \Sigma(Y_j)$.
Then we can put $\{X_1, \ldots, X_m, Y_1, \ldots, Y_n\}$ in very good position.
\end{thm}

To prove Theorem~\ref{MIXING}, we will apply the arguments laid out in Section~\ref{VeryGoodPosition},
with a few minor modifications.

\begin{lem}
Let $\Sigma$ be as in Theorem~\ref{MIXING}.
The relation $\leq$ on $\Sigma$ given by:
\[
	A \leq B \iff A \subset B \text{ or } A \cap B^* \text{ is the only small corner of the pair } (A,B)
\]
is well-defined, and is a partial order on $\Sigma$.
\end{lem}

\begin{proof}
Since $G$ is finitely generated, a corner of the pair $(A,B)$ is $Stab(A)$-finite if, and only if,
the corner is $Stab(B)$-finite (proved in Lemma 2.3 of \cite{Scott1998}).
Hence ``smallness'' of a corner is well-defined.

Next we claim that if a pair $(A,B)$ has two small corners, then one is empty.
Suppose two corners of $(A,B)$ are small.
By Lemma~\ref{COMMENSURABLE}, $Stab(A)$ and $Stab(B)$ are commensurable.
Since finite index subgroups of finitely generated groups are finite generated,
and since we are assuming that the $H_i$'s are finitely generated and the $K_j$'s are not finitely generated,
this implies either $A, B \in \mycup \limits_{i =1}^m \Sigma(X_i)$ or $A, B \in \mycup \limits_{i =1}^n \Sigma(Y_j)$.
In the first case, since $\mycup \limits_{i =1}^m \Sigma(X_i)$ is in very good position,
we must have an empty corner of the pair $(A,B)$.
In the second case, since no two $Y_j$'s yield isomorphic splittings,
we must have $A,B \in \Sigma(Y_j)$ for the same $j$, and hence one corner of $(A,B)$ is empty.

Finally we show that $\leq$ defines a partial order on $\Sigma$.
Since $\mycup \limits_{i =1}^m \Sigma(X_i)$ is in very good position,
The relation $\leq$, when restricted to $\mycup \limits_{i =1}^m \Sigma(X_i)$, is identical to inclusion.
Now the proof of~\ref{PARTIALORDER} shows that $\leq$ defines a partial order on all of $\Sigma$.
\end{proof}

\begin{lem}
Let $\Sigma$ be as in Theorem~\ref{MIXING}.
For all $A, B \in \Sigma$, there are only finitely many $C \in \Sigma$ such that $A \leq C \leq B$.
\end{lem}

\begin{proof}
Since $G$ and all the $H_i$ are finitely generated,
Lemma 1.15 of \cite{ScottSwarup2000}
shows that for all $A, B \in \Sigma(X_i)$,
there are only finitely many $C \in \Sigma(X_i)$ such that $A \leq C \leq B$.
Since each $Y_j$ is a standard almost invariant set arising from a splitting,
for all $A, B \in \Sigma(Y_j)$,
there are only finitely many $C \in \Sigma(X_i)$ such that $A \leq C \leq B$.

By the proof of Proposition~\ref{INTERVALS}, for all $A, B \in \Sigma$,
only finitely many $C \in \Sigma$ satisfy $A \leq C \leq B$.
\end{proof}

To prove~\ref{MIXING},
apply the construction laid out in Section~\ref{VeryGoodPosition}.
The only other modification needed is to note (for example, in the proof of the first part of Lemma~\ref{CONNECTED})
that since $\mycup \limits_{i=1}^{m} \Sigma(X_i)$ is in very good position, each $\Sigma(X_i)$ is nested.


\bibliographystyle{amsplain}
\bibliography{BIBLIO}

\providecommand{\MR}[1]{}
\providecommand{\bysame}{\leavevmode\hbox to3em{\hrulefill}\thinspace}
\providecommand{\MR}{\relax\ifhmode\unskip\space\fi MR }
\providecommand{\MRhref}[2]{%
  \href{http://www.ams.org/mathscinet-getitem?mr=#1}{#2}
}
\providecommand{\href}[2]{#2}
\begin{thebibliography}{10}

\bibitem{Cohen1970}
Daniel~E. Cohen, \emph{Ends and free products of groups}, Math. Z. \textbf{114}
  (1970), 9--18. \MR{0260877 (41 \#5497)}

\bibitem{Dunwoody1979}
M.~J. Dunwoody, \emph{Accessibility and groups of cohomological dimension one},
  Proc. London Math. Soc. (3) \textbf{38} (1979), no.~2, 193--215. \MR{531159
  (80i:20024)}

\bibitem{Freudenthal1931}
Hans Freudenthal, \emph{\"{U}ber die {E}nden topologischer {R}\"aume und
  {G}ruppen}, Math. Z. \textbf{33} (1931), no.~1, 692--713. \MR{1545233}

\bibitem{HNN1949}
Graham Higman, B.~H. Neumann, and Hanna Neumann, \emph{Embedding theorems for
  groups}, J. London Math. Soc. \textbf{24} (1949), 247--254. \MR{0032641
  (11,322d)}

\bibitem{Hopf1944}
Heinz Hopf, \emph{Enden offener {R}\"aume und unendliche diskontinuierliche
  {G}ruppen}, Comment. Math. Helv. \textbf{16} (1944), 81--100. \MR{0010267
  (5,272e)}

\bibitem{Houghton1974}
C.~H. Houghton, \emph{Ends of locally compact groups and their coset spaces},
  J. Austral. Math. Soc. \textbf{17} (1974), 274--284, Collection of articles
  dedicated to the memory of Hanna Neumann, VII. \MR{0357679 (50 \#10147)}

\bibitem{NibloSageevScottSwarup2005}
Graham Niblo, Michah Sageev, Peter Scott, and Gadde~A. Swarup, \emph{Minimal
  cubings}, Internat. J. Algebra Comput. \textbf{15} (2005), no.~2, 343--366.
  \MR{2142089 (2006h:20056)}

\bibitem{Niblo2004}
Graham~A. Niblo, \emph{A geometric proof of {S}tallings' theorem on groups with
  more than one end}, Geom. Dedicata \textbf{105} (2004), 61--76. \MR{2057244
  (2005e:20060)}

\bibitem{Sageev1995}
Michah Sageev, \emph{Ends of group pairs and non-positively curved cube
  complexes}, Proc. London Math. Soc. (3) \textbf{71} (1995), no.~3, 585--617.
  \MR{1347406 (97a:20062)}

\bibitem{Schreier1927}
Otto Schreier, \emph{Die untergruppen der freien gruppen}, Abh. Mat Sem. Univ.
  Hamburg \textbf{3} (1927), 167--169.

\bibitem{Scott1998}
Peter Scott, \emph{The symmetry of intersection numbers in group theory}, Geom.
  Topol. \textbf{2} (1998), 11--29 (electronic). \MR{1608688 (99k:20076a)}

\bibitem{ScottSwarup2000}
Peter Scott and Gadde~A. Swarup, \emph{Splittings of groups and intersection
  numbers}, Geom. Topol. \textbf{4} (2000), 179--218 (electronic). \MR{1772808
  (2001h:20032)}

\bibitem{ScottSwarup2003}
\bysame, \emph{Regular neighbourhoods and canonical decompositions for groups},
  Ast\'erisque (2003), no.~289, vi+233. \MR{2032389 (2005f:20045)}

\bibitem{ScottSwarup2003errata}
\bysame, \emph{Errata for ``regular neighbourhoods and canonical decompositions
  for groups''},
  http://www.math.lsa.umich.edu/{\textasciitilde}pscott/preprints.html, October
  2006.

\bibitem{ScottWall1979}
Peter Scott and Terry Wall, \emph{Topological methods in group theory},
  Homological group theory ({P}roc. {S}ympos., {D}urham, 1977), London Math.
  Soc. Lecture Note Ser., vol.~36, Cambridge Univ. Press, Cambridge, 1979,
  pp.~137--203. \MR{564422 (81m:57002)}

\bibitem{Serre1977}
Jean-Pierre Serre, \emph{Arbres, amalgames, {${\rm SL}\sb{2}$}}, Soci\'et\'e
  Math\'ematique de France, Paris, 1977, Avec un sommaire anglais,
  R{\'e}dig{\'e} avec la collaboration de Hyman Bass, Ast{\'e}risque, No. 46.
  \MR{0476875 (57 \#16426)}

\bibitem{Serre1980}
\bysame, \emph{Trees}, Springer-Verlag, Berlin, 1980, Translated from the
  French by John Stillwell. \MR{607504 (82c:20083)}

\bibitem{Specker1950}
Ernst Specker, \emph{Endenverb\"ande von {R}\"aumen und {G}ruppen}, Math. Ann.
  \textbf{122} (1950), 167--174. \MR{0038984 (12,479g)}

\bibitem{Stallings1968}
John~R. Stallings, \emph{On torsion-free groups with infinitely many ends},
  Ann. of Math. (2) \textbf{88} (1968), 312--334. \MR{0228573 (37 \#4153)}

\bibitem{Stallings1971}
\bysame, \emph{Group theory and {$3$}-manifolds}, Actes du {C}ongr\`es
  {I}nternational des {M}ath\'ematiciens ({N}ice, 1970), {T}ome 2,
  Gauthier-Villars, Paris, 1971, pp.~165--167. \MR{0425972 (54 \#13921)}

\bibitem{Swan1969}
Richard~G. Swan, \emph{Groups of cohomological dimension one}, J. Algebra
  \textbf{12} (1969), 585--610. \MR{0240177 (39 \#1531)}

\end{thebibliography}

\end{document}